\UseRawInputEncoding

\documentclass[12pt]{amsart}
\usepackage{amssymb}
\usepackage[all]{xy}
\usepackage[toc,page,title,titletoc,header]{appendix}
\usepackage{xcolor}

\theoremstyle{plain}
\newtheorem{thm}{Theorem}[section]

\newtheorem*{thm1.1}{Theorem 1.1}

\newtheorem{lemma}[thm]{Lemma}
\newtheorem{lem}[thm]{Lemma}
\newtheorem{cor}[thm]{Corollary}
\newtheorem{pro}[thm]{Proposition}

\theoremstyle{definition}

\newtheorem{rem}[thm]{Remark}

\newtheorem{defi}[thm]{Definition}
\newtheorem{exe}[thm]{Example}

\newtheorem{que}[thm]{Question}
\newtheorem{con}[thm]{Conjecture}
\newtheorem{fact}[thm]{Fact}

\numberwithin{equation}{section}
\newcounter{elno}                % This to number lists

\newcommand{\la}{\lambda}

\newcommand{\an}{{\rm an}}

\newcommand{\Pic}{{\rm Pic}}
\newcommand{\CH}{{\rm CH}}

\newcommand{\Spec}{{\rm Spec \,}}

\newcommand{\Frac}{{\rm Frac \,}}

\newcommand{\Char}{{\rm char}}

\newcommand{\id}{{\rm id}}

\newcommand{\Per}{{\rm Per}\,}

\renewcommand{\d}{\stackrel{\mbox{\scriptsize{$\bullet$}}}{}}
\newcommand{\boxtensor}{{\Box\kern-9.03pt\raise1.42pt\hbox{$\times$}}}

\newcommand{\propsubset}
{\mbox{$\textstyle{
			\subseteq_{\kern-5pt\raise-1pt\hbox{\mbox{\tiny{$/$}}}}}$}}
\newcommand{\sL}{{\mathcal L}}

\newcommand{\sO}{{\mathcal O}}

\newcommand{\sV}{{\mathcal V}}

\newcommand{\sX}{{\mathcal X}}

\newcommand{\sZ}{{\mathcal Z}}
\newcommand{\A}{{\mathbb A}}

\newcommand{\bB}{{\mathbf{B}}}
\newcommand{\C}{{\mathbb C}}

\renewcommand{\P}{{\mathbb P}}
\newcommand{\Q}{{\mathbb Q}}
\newcommand{\R}{{\mathbb R}}

\newcommand{\Z}{{\mathbb Z}}
\newcommand{\bk}{{\mathbf{k}}}

\begin{document}
	%%%%%%%%%%%%%%%%%%%%%%%%%%%%%%%%%%%%%%%%%%%%%%%%%%%%%%%%%%%%%%
	%%%%%%%%%%%%%%%%%%%%%%%%%%%%%%%%%%%%%%%%%%%%%%%%%%%%%%%%%%%%%%
	\title[]{Algebraic dynamics and recursive inequalities}
	
	\author{Junyi Xie}
	
	\address{Beijing International Center for Mathematical Research, Peking University, Beijing 100871, China}
	
	%\email{junyi.xie@univ-rennes1.fr}
	\email{xiejunyi@bicmr.pku.edu.cn}

	\thanks{The author is supported by the NSFC Grant No.12271007.}

	\date{\today}

	\bibliographystyle{alpha}

	%\subjclass[2010]{
		%Primary:
		%37P55, %arithmetic dynamics on general algebraic varieties
		%37P30.
		%Secondary:
		%14J50. %Automorphisms of surfaces and higher-dimensional varieties
		%%08A35, %Automorphisms, endomorphisms
		%%32H50, %iteration problem,
		%%37B40, %Topological entropy
		%%11G10, %Abelian varieties of dimension > 1
		%%14M25, %Toric varieties
		%%14H30, %Coverings, fundamental group
		%%14E30, %Minimal model program (Mori theory, extremal rays)
		%%20K30, %(20Kxx: Abelian groups:) Automorphisms, homomorphisms, endomorphisms, etc.
		%%32M05, %Complex Lie groups, automorphism groups acting on complex spaces
		%%11G10, %Abelian varieties of dimension >1
		%}
	
	\maketitle
	
	\centerline{\it In memory of the late Professor Nessim Sibony}

	\begin{abstract}
We get three basic  results in algebraic dynamics:
	(1). We give the first algorithm to compute the dynamical degrees to arbitrary precision.
	(2). We prove that for a family of dominant rational self-maps, the dynamical degrees are lower semi-continuous with respect to the Zariski topology. This implies a conjecture of Call and Silverman. 
	(3). We prove that the set of periodic points of a cohomologically hyperbolic rational self-map is Zariski dense. 
	
	Moreover, we prove the Kawaguchi-Silverman conjecture for a class of self-maps of projective surfaces including all the birational ones.
	
	In fact, for every dominant rational self-map, we find a family of recursive inequalities of some dynamically meaningful cycles. Our proofs are based on these  inequalities.
	\end{abstract}
	
	\tableofcontents
	
	\section{Introduction}

Let $\bk$ be a field. Let $X$ be a projective variety of dimension $d$ over $\bk.$ Let $f: X\dashrightarrow X$ be a dominant rational self-map. 
The aim of algebraic dynamics is to study algebraic and dynamical properties of the iterates
of $f$.

%\medskip
\subsection{Dynamical degrees}\label{subsecdynamicaldeg}
The most fundamental dynamical invariants associated to an algebraic dynamical system is arguably its dynamical degrees.

Let $L$ be a big and nef line bundle in $\Pic(X)$.  Denote by $I(f)$ the indeterminacy locus of $f$.
Let $X'$ be the graph of $f$ in $X\times X$ i.e. the Zariski closure of $\{(x,f(x))|\,\, x \text{ is a closed point in } X\setminus I(f)\}$ and let $\pi_1$, $\pi_2$ be the projection to the first and the second factors. For $i=0,\dots, d$, the \emph{$i$-th degree} of $f$ is $$\deg_{i,L}f:=((\pi_2^*L)^i\cdot (\pi_1^*L)^{d-i}).$$
Using the terminology from Section \ref{subsectioncocycles} and \ref{subseclinebun}, we may write $$\deg_{i,L}f=((f^*L)^i\cdot L^{d-i})$$ without specifying the birational model $X'.$
The $i$-th \textit{dynamical degree} of $f$ is
$$
	\la_i(f):=\lim_{n\to\infty}(\deg_{i,L} f^n)^{1/n}\geq 1.
$$
The existence of the above limit is non-trivial. 
It was proved by Russakovskii and Shiffman \cite{Russakovskii1997} when $X=\P_{\C}^d$, and by Dinh and Sibony \cite{Dinh2005}
when $X$ is projective over $\C$. As shown in \cite{Dinh2004} by Dinh and Sibony, the dynamical degrees can be defined even for meromorphisms on K\"ahler manifolds.
It was proved by Truong \cite{Truong2020} and Dang \cite{Dang2020} in arbitrary characteristic. The methods of Truong \cite{Truong2020} and Dang \cite{Dang2020} are different. 
Truong's method is based on Jong's alterations and Roberts' effective version of Chow's moving lemma. This method can be viewed as an algebraic mimic of Dinh-Sibony's proof using positively closed currents \cite{Dinh2004,Dinh2005}. Dang's method is based on Siu's inequality. 
The definition of $\la_i(f)$ does not depend on the choice of $L$ \cite{Dinh2004,Dinh2005,Truong2020,Dang2020}.
Moreover, if $\pi: X\dashrightarrow Y$ is a generically finite and dominant rational map between varieties and $g\colon Y\dashrightarrow Y$ is a rational self-map such that $g\circ\pi=\pi\circ f$, then $\la_i(f)=\la_i(g)$ for all $i$. This can be shown by combining \cite[Theorem 1]{Dang2020} with the projection formula.
Another way to prove it is to apply the product formula for relative dynamical degrees directly (c.f. \cite{Dinh2011}, \cite{Dang2020} and \cite[Theorem 1.3]{Truong2020}).

\medskip

Roughly speaking, the dynamical degrees measure the algebraic complexity of $f$. It controls the topological complexity of $f$.
When $X$ is a smooth projective variety over $\C$ and $f$ is an endomorphism, fundamental results of Gromov \cite{Gromov2003} and Yomdin  \cite{Yomdin1987} 
show that  $$h_{top}(f^{\an})=\max_{0\leq i\leq d}\{\la_i(f)\},$$
where $h_{top}(f^{\an})$ is the topological entropy of the holomorphic endomorphism $f^{\an}: X(\C)\to X(\C)$ induced by $f.$
Dinh-Sibony \cite{Dinh2005} showed that the upper bound 
\begin{equation}\label{equrationads}h_{top}(f^\an)\leq \max_{0\leq i\leq d}\{\la_i(f)\}
\end{equation} still holds for arbitrary rational self-maps over $\C.$
However, (\ref{equrationads}) can be strict in general \cite{Guedj2005a}.
Recently Favre, Truong and the author proved (\ref{equrationads}) in the non-archimedean case \cite{Favre2022}.
In the non-archimedean case,  (\ref{equrationads}) can be strict even for endomorphisms \cite{Favre2010,Favre2022}.

 When $\bk=\overline{\Q},$ the dynamical degrees also control the arithmetic complexity of $f$, which is measured by the notion of arithmetic degree (c.f. Section \ref{sectionksc}).
Further, the Kawaguchi-Silverman conjecture (=Conjecture \ref{KSCsurfacelaonela}) asserts that for any point $x\in X(\bk)$ with Zariski dense orbit, the arithmetic degree $\alpha_f(x)$ for $(X,f,x)$ equals $\la_1(f).$ This conjecture has attracted a lot of attention.
For the recent development, see \cite{Matsuzawa2023} and the references therein.
See \cite{Song2023,Luo2023} for the higher arithmetic degrees and their relations to the higher dynamical degrees.
In Section \ref{sectionksc}, we will prove the Kawaguchi-Silverman conjecture for a class of self-maps of projective surfaces including all the birational ones.

When $f$ is an endomorphism, the dynamical degrees control the action of $f^*$ on the cohomology of $X$. 
When $X$ is a smooth projective variety over $\C$, Dinh \cite{Dinh2005c} proved that 
\begin{equation}\label{equalairho}\la_i(f)=\rho(f^*:H^{2i}(X(\C),\R)\to H^{2i}(X(\C),\R))
\end{equation}
where $H^{2i}(X(\C),\R)$ is the singular cohomology of degree $2i$ and $\rho(f^*)$ is the spectral radius of the linear operator $f^*$. 
In positive characteristic, Truong proposed a conjecture saying that (\ref{equalairho}) still holds
if one replaces the singular cohomology by the $\Q_l$-cohomology with $l\neq \Char\, \bk$ \cite{Truong2016}.
This conjecture is wildly open. Indeed, the case for Frobenius endomorphisms implies Deligne's famous theorem for Weil's Riemann hypothesis \cite{Deligne1974}.
However, it was proved by Esnault and Srinivas \cite{Esnault2013} for surface automorphisms and by Truong \cite{Truong2016} for any dominant endomorphisms of smooth projective varieties, that $$\max_{0\leq i\leq d}\la_i(f)=\max_{0\leq i\leq 2d}\rho(f^*: H^{i}(X, \Q_{l})\to H^{i}(X, \Q_{l}))$$  with respect to any field embedding $\Q_l\hookrightarrow \C.$
Truong's proof indeed relies on Deligne's theorem.

%\subsubsection{Cohomologically hyperbolic self-maps}
%We introduce the notion of cohomological Lyapunov exponents as follows:
%For $i=1,\dots, d$, define the $i$-th \emph{cohomological Lyapunov exponent} of $f$ to be $$\mu_i(f):=\la_i(f)/\la_{i-1}(f).$$
%Define $\mu_{d+1}(f):=0$ for convenience. 
%As the sequence of dynamical degrees is log-concave \cite{Dinh2005, Truong2020,Dang2020}, the sequence $\mu_i(f), i=1\dots, d$ is decreasing. 
%

\subsubsection{Cohomologically hyperbolic self-maps}
The dynamical degree $\la_i(f)$ can be thought as the volume growth rate of a ```random" $i$-dimensional subvariety. So one may think $\la_i(f)/\la_{i-1}(f)$ as the ``speed of the $i$-th fastest direction". For this reason, we introduce the notion of cohomological Lyapunov exponents as the algebraic analogy of the Lyapunov exponents
as follows:
For $i=1,\dots, d$, define the $i$-th \emph{cohomological Lyapunov exponent} of $f$ to be $$\mu_i(f):=\la_i(f)/\la_{i-1}(f).$$
Define $\mu_{d+1}(f):=0$ for convenience. 
As the sequence of dynamical degrees is log-concave \cite{Dinh2005, Truong2020,Dang2020}, the sequence $\mu_i(f), i=1\dots, d$ is decreasing.

\medskip

For $i=1,\dots,d$, we say that $f$ is \emph{$i$-cohomologically hyperbolic} if $\la_i(f)$ is strictly larger than other dynamical degrees i.e. $$\mu_i(f)>1 \text{ and } \mu_{i+1}(f)<1.$$
We say that $f$ is \emph{cohomologically hyperbolic} if it is $i$-cohomologically hyperbolic for some $i=1,\dots,d$, in other words, $\mu_j(f)\neq 1$ for every $j=1,\dots,d.$

Cohomological hyperbolicity can be viewed as a cohomological version of the important notion of hyperbolic dynamics in differentiable dynamical systems. Indeed, when $\bk=\C$, 
very few algebraic dynamical system could be Anosov (which is a strong version of hyperbolicity) c.f. \cite{Ghys1995,Cantat2004,Xu2024}.
However people expect that a cohomologically hyperbolic self-map looks like a hyperbolic map, hence shares some properties of hyperbolic maps.

\subsection{Algorithm to compute the dynamical degrees}
A basic problem is to compute the dynamical degrees to any given precision.
More precisely,
\begin{que}\label{quecompdyna}
 For any given number $l\in \Z_{>0},$  is there an algorithm (that stops in finite time) to compute a number $\tilde{\la}$ such that $\la_i(f)\in (\tilde{\la}, \tilde{\la}+\frac{1}{2^l})$?
\end{que}

Let $L$ be an ample (or big and nef) line  bundle on $X$.
By the definition of the dynamical degree, for $n$ sufficiently large,  we have $$\la_i(f)\in ((\deg_{i,L} f^n)^{1/n}-\frac{1}{2^{l+1}}, (\deg_{i,L} f^n)^{1/n}+\frac{1}{2^{l+1}}).$$
But this does not answer Question \ref{quecompdyna}, as we do not know how large $n$ we need.

\medskip

Question \ref{quecompdyna} is also interesting in cryptography.  See \cite[Section 2]{Shepherd-Barron2021} for interesting discussions.

\subsubsection{Our result}
Strictly speaking, the answer to Question \ref{quecompdyna} depends on the input i.e. 
how we represent $X$ and $f$.
\begin{exe}\label{exehalting} Let $X:=\P^1_{\C}$. Define $f_n: X\to X, n\geq 1$ as follows:
Let $T_n, n\geq 0$ be all the Turing machines. Define $a_n:=0$ if $T_n$ will halt, and $a_n:=1$ if $T_n$ will not halt.
Define $f_n(z):=a_nz^2+z$. As $\la_1(f_n)=\deg f_n$, $\la_1(f_n)=1$ if $T_n$ will halt, and $\la_1(f_n)=2$ if $T_n$ will not halt.
As the Halting problem is unsolvable, there is no algorithm to compute $\la_1(f_n)$ for all $n\geq 0.$
\end{exe}

To describe our input, we need the notion of \emph{mixed degrees}:
Let $L$ be an ample (or big and nef) line bundle on $X$.
Let $s\geq 1$, consider two sequence of non-negative integers: $m_1> \dots >m_s\geq 0$ and $r_1,\dots, r_s\geq 0$ with $\sum_{i=1}^sr_i=d.$
The  mixed degree $(L_{m_1}^{r_1}\cdots L_{m_s}^{r_s})$ is easier to define using the terminology in Section \ref{subsectioncocycles} and \ref{subseclinebun}.
Here we define it in a more direct way.
Let $X'$ be the graph in $X^{s+1}=X\times (X^s)$ of the morphism $X\to X^s$ sending $x$ to $(f^{m_1}(x),\dots,f^{m_s}(x)).$
Let $\pi_i$ be the projection to the $(i+1)$-th factor. 
Then we define $$(L_{m_1}^{r_1}\cdots L_{m_s}^{r_s}):=((\pi_1^*L)^{r_1}\cdots (\pi_s^*L)^{r_s})\in \Z_{>0}.$$

\medskip

We note that $(X,f,L)$ can be defined on a finitely generated field.
The following remark shows that the mixed degrees are computable if we represent $(X,f)$ in a reasonable form.
\begin{rem}\label{remfinitelygef} 
Assume that $\bk$ is a finitely generated field. 
We represent $X$ and $f$ as follows:
Write $\bk$ as $$\bk:=\Frac(\Z[t_1,\dots,t_l]/P)$$ where $P=(G_1,\dots, G_m)$ is a prime ideal of $\Z[t_1,\dots,t_l].$
Write $X$ as the subvariety of $\P^N_\bk$ defined by a homogenous prime ideal $(H_1,\dots,H_s)$. 
The rational map $f: X\dashrightarrow X$ extends to a rational self-map $F: \P^N\dashrightarrow \P^N$ sending $[x_0:\dots: x_N]$ to $[F_0:\dots: F_N]$ where $F_0,\dots, F_N$ are homogenous polynomials of the same degree in $\bk[x_0,\dots, x_N].$

We represent $X,f$ using the following datas as inputs:
\begin{points}
\item  the polynomials $G_1,\dots, G_m$ with integer coefficients; 
\item the polynomials $H_1,\dots,H_s, F_0,\dots, F_N$, whose coefficients are  represented as rational functions in $t_1,\dots,t_l$ with integer coefficients.
\end{points}

In this case,  we may ask $L$ to be the restriction of $\sO_{\P^N}(1)$ to $X.$ For every mixed degree $(L_{m_1}^{r_1}\cdots L_{m_s}^{r_s})$, there is an algorithm 
compute its exact value. So our assumption is satisfied and we may use  Theorem \ref{thmcomputela} to compute the dynamical degrees of $f.$
\end{rem}

In Section \ref{sectionalgdeg},
we affirmatively answer Question \ref{quecompdyna}, under the assumption that all the mixed degrees $(L_{m_1}^{r_1}\cdots L_{m_s}^{r_s})$ are computable.
\begin{thm}\label{thmcomputela} For any given number $l\in \Z_{>0},$ there is an explicit algorithm to output numbers $\widetilde{\la_i}, i=0,\dots, d$ such that $\la_i\in (\widetilde{\la_i}, \widetilde{\la_i}+\frac{1}{2^l})$, using finitely many mixed degrees.
\end{thm}

\subsubsection{Previous results}
Here we ignore the difficulty from the computability theory as in Example \ref{exehalting}, and we assume that $(X,f)$ is represented in a reasonable form.

\medskip
%Most of the previous study focus on the first dynamical degree $\la_1(f)$.

There are plenty of works concerning the computation of dynamical degrees in special cases.

If $X$ is a smooth projective variety and $f: X\to X$ is an endomorphism, we have that $$\la_i(f)=\rho(f^*: N^i(X)_{\R}\to N^i(X)_{\R})$$ where $N^i(X)_{\R}$ is the $\R$-vector space spanned by the numerical classes of $i$-cocycles of $X.$ As $N^i(X)_{\R}$ is a finite dimensional vector space,  the sequence $\deg_{i,L} f^n, n\geq 0$ satisfies a linear recursive equation of order $\leq \dim_{\R}N^i(X)_{\R}.$ In particular, $\la_i(f)$ is an algebraic integer of degree $\leq \dim_{\R}N^i(X)_{\R}.$
In this case, $\la_i(f)$ should be computable for a given $(X,f)$.   
On the other hand, there is a lot of  interesting works on constructing examples of endomorphisms (especially automorphisms) having certain properties on the dynamical degrees, e.g. \cite{Cantat1999,McMullen2002,McMullen2007,McMullen2011,McMullen2016,Dolgachev2018,Oguiso2010,Oguiso2014,Cantat2015,Oguiso2009,Oguiso2015,Oguiso2020,Uehara2016,Reschke2017,Lesieutre2021a}.

When $f$ is merely rational, most of the previous work focus on the first dynamical degree $\la_1(f)$.
In \cite{Sibony1999}, Sibony introduced the important notion of algebraically stable maps. 
If $(X,f)$ is algebraically stable, as in the endomorphism case, we still have 
$$\la_1(f)=\rho(f^*: N^1(X)\to N^1(X)),$$ and one can compute $\la_1(f)$ using linear algebra. 
In most of the works, the strategy to compute $\la_1(f)$ is to construct a birational model $(X',f')$ of $(X,f)$ for which $(X',f')$ is algebraicaly stable.
For certain classes of maps, such as birational self-maps of surfaces \cite{favre} or endomorphism of $\A^2$ \cite{Favre2007,Favre2011}, we can find such birational models, after a suitable iterate. 
On the other hand, it was proved by Favre \cite{Favre2003} that algebraically stable model may not exist even for monomial self-maps on $\P^2$.
However, the dynamical degrees $\la_i(f), i=0,\dots, d$ are computed for monomial self-maps on $\P^N$ by Favre-Wulcan and Lin \cite{Favre2012,Lin2012}. 
Dinh and Sibony \cite{Tien-CuongDinh} computed the dynamical degrees for automorphisms $f : \mathbb{A}^m \to \mathbb{A}^m$ on complex affine spaces that are regular (i.e. the indeterminacy loci of $f$ and its inverse $f^{-1}$ are disjoint on the hyperplane at infinity in $\mathbb{P}^m$).
As far as we know, these are the only non-trivial cases for which higher dynamical degrees are computed.  See \cite{Bedford2004,Bedford2008,Nguyen2006,AnglesdAuriac2006,Abarenkova1999,Bellon1999,M.Henneaux1997,favre,Favre2011,Bell2020a,Dang2021,Bell2023} and the reference therein for more related works.

In the above cases, the dynamical degrees are always algebraic. Moreover, as $(X,f)$ can be defined on a finitely generated field, the set of all possible values of dynamical degrees are countable \cite{Bonifant2000,Urech}.
However, in the breakthrough work \cite{Bell2020a}, Bell, Diller and Jonsson give examples of rational self-maps of $\P^2_{\C}$ whose first dynamical degrees are transcendental numbers. In their examples, $\la_1(f)$ is the unique positive real zero of certain transcendental power series.  So we do not expect the existence of an algorithm (stoping in finite time) computing the exact value of $\la_i(f)$ in  general.

\subsubsection{The surface case}
In \cite[Key Lemma]{Xie2015}, the author proved the following result.\footnote{In \cite[Key Lemma]{Xie2015}, the result is stated only for birational self-maps. However, its proof indeed works for any dominant rational self-map,  replacing all $f^*$ by $\frac{f^*}{\la_2^{1/2}}$.} 
\begin{thm}\label{thmxiedukeklin}
Let $\bk$ be a field. Let $X$ be a projective surface over $\bk.$ Let $f: X\dashrightarrow X$ be a dominant rational self-map. Let $L$ be a big and nef line bundle. Then we have $$\la_1(f)\geq \frac{\deg_{1,L} f^2}{2^{\frac{1}{2}}\times 3^{18}\deg_{1,L} f}.$$
\end{thm}
We will see, in Section \ref{subsecdimtwo}, that this  indeed implies a positive answer of Question \ref{quecompdyna} for rational self-maps of surfaces.

The proof of Theorem \ref{thmxiedukeklin} relies on the theory of hyperbolic geometry and the natural linear action of $f$ on a suitable hyperbolic space of infinite dimension. This space is constructed as a set of cohomology classes in the Riemann-Zariski space of $X$ and was introduced by Cantat \cite{Cantat2007}. Unfortunately, such a space can only  be constructed in dimension two. Also the coefficient $2^{\frac{1}{2}}\times 3^{18}$ is quite large. 
In this paper, we give a new proof of Theorem \ref{thmxiedukeklin} with a better coefficient i.e. from $2^{\frac{1}{2}}\times 3^{18}$ to $4$.
\begin{thm}\label{thmsurfaceboundin}Let $\bk$ be a field. Let $X$ be a projective surface over $\bk.$ Let $f: X\dashrightarrow X$ be a dominant rational self-map. Let $L$ be a big and nef line bundle. Then we have $$\la_1(f)\geq \frac{\deg_{1,L} f^2}{4\deg_{1,L} f}.$$
\end{thm}
The proof of Theorem \ref{thmsurfaceboundin} does not rely on hyperbolic geometry and is much simpler (c.f. Section \ref{subsecdimtwo}).

\subsection{Lower semi-continuity of dynamical degrees}
Besides the the dynamical degrees of a single map, we also study the behavior of the dynamical degree in families.

Let $S$ be an integral noetherian scheme and $d\in \Z_{\geq 0}.$
Initially, a family of self-maps on $S$ should be a collection of dominant rational self-maps $f_p: X_p\dashrightarrow X_p, p\in S$ on varieties $p\in S.$ 
So we introduce the following definition.
\begin{defi}A \emph{family of $d$-dimensional dominant rational self-maps on $S$} is a flat and projective scheme $\pi:\sX\to S$ satisfying $\dim \sX/S=d$ with  a dominant rational self-map
$f: \sX\dashrightarrow \sX$ over $S$ such that for every $p\in S$,
\begin{points}
\item  the fiber $X_p$ of $\pi$ at $p$ is geometrically reduced and irreducible;
\item  $X_p\not\subseteq I(f)$;
\item  the induced map $f_p: X_p\dashrightarrow X_p$ is dominant. 
\end{points}
\end{defi}

%\rem
%The above definition is reasonable. Initially, a family of self-maps on $S$ should be a collection of self-maps $f_p: X_p\dashrightarrow X_p, p\in S.$
%Condition (i) means that $X_p, p\in S$ are varieties. Condition (ii) means that $f_p$ are well-defined rational maps and Condition (iii) means that the $f_p$ are dominant. 
%\endrem

\medskip

We prove the following result in Section \ref{sectionsslsc}.
%In the following result, we prove the lower semi-continuity of dynamical degrees for a family of dominant rational self-maps.
\begin{thm}\label{thmlscdynadegin}
Let $S$ be an integral noetherian scheme and $\pi:\sX\to S$ be a flat and projective scheme over $S$ with $\dim \sX/S=d$. 
Let $f: \sX\dashrightarrow \sX$ be a family of $d$-dimensional dominant rational self-maps on $S$. Then
for every $i=0,\dots,d$, the function $p\in S\mapsto \lambda_i(f_p)$ is lower semi-continuous in the Zariski topology on $S$.
\end{thm}

A special case of Theorem \ref{thmlscdynadeg} is the following result.
\begin{cor}\label{corlscdynadegoverzin}
Let $f: \P^d_{\Z}\dashrightarrow \P^d_{\Z}$ be a dominant rational self-map over $\Z$. 
Then for every $i=0,\dots,d$, we have 
$$\la_i(f\otimes_{\Z}\Q)=\lim_{p \text{ prime}, p\to \infty} \la_i(f_p).$$
\end{cor}

Theorem \ref{thmlscdynadegin} generalizes \cite[Theorem 4.3]{Xie2015} from dimension two to any dimension.
The special case where $i=1$ and $\sX=\P^N_S$ of Theorem \ref{thmlscdynadegin} implies Call-Silverman's conjecture \cite[Conjecture 1]{Silverman2018} and its generalized version \cite[Conjecture 14.13]{Benedetto2019}. Corollary \ref{corlscdynadegoverzin} gives a positive answer to \cite[Question 14.10]{Benedetto2019}.

In \cite[Section 4.3]{Xie2015}, the author  provided the following example showing that Corollary \ref{corlscdynadegoverzin} cannot be strengthened to the statement that 
$\la_i(f\otimes_{\Z}\Q)=\la_i(f_p)$ for infinitely many prime $p$.
\exe
Let $f: \P^2_{\Z}\dashrightarrow \P^2_{\Z}$ be the rational self-map sending $[x:y:z]$ to $[xy: xy-2z^2: yz+3z^2].$
 Then $\la_1(f\otimes_{\Z}\Q)=2$, but $\la_1(f_p)<2$ for all primes $p.$
\endexe

\subsubsection{Strategy of the proof}
There are three steps in the proof. In the first step, we get a simple criterion for lower semi-continuity functions on noetherian schemes (c.f. Lemma \ref{lemlsconneoth}).
Next we show that the mixed degrees are lower semi-continuous (c.f. Lemma \ref{lowersemicontinuous}). 
This step can be shown using our criterion Lemma \ref{lemlsconneoth}, Raynaud-Gruson flattening theorem \cite[Theorem 5.2.2]{Raynaud1971}, and the constancy of intersection numbers on flat families \cite[Proposition 10.2]{Fulton1984}. 
The function $p\in S\mapsto \lambda_i(f_p)$ of the $i$-th dynamical degree is the point-wise limit of the functions $p\in S\mapsto (\deg_{i,L_p}f_p)^{1/n}$.
By the second step, the later function is lower semi-continuous.
However, as shown in Remark \ref{remlimlsc}, limit of lower semi-continuous functions may not be lower semi-continuous.
To complete the proof of Theorem \ref{thmlscdynadeg}, we need to show that the dynamical degrees are continuous at the generic point of $S$.
In this step, the main ingredient is the lower bounds of the dynamical degrees obtained in Section \ref{secrecurdeg}.

\subsection{Periodic points}

One of the most basic problem in algebraic dynamics is to determine when the set of periodic points is Zariski dense. 
\begin{que}\label{quewhenperzar}
Under which condition does $f$ admit  a Zariski dense set of periodic points?
\end{que}

We give a positive answer to Question \ref{quewhenperzar} for cohomologically hyperbolic self-maps in Section \ref{sectioncohhyp}.

\begin{thm}\label{thmperiodisodenseint} If $f$ is cohomologically hyperbolic i.e. there is a unique $i\in \{1,\dots, d\}$ such that $\la_i(f)=\max_{j=0,\dots, d}\la_j(f),$ then the set of periodic closed points is Zariski dense. 
\end{thm}
We indeed prove a stronger statement in Theorem \ref{thmperiodisodense}.

\medskip

When $f$ is not cohomologically hyperbolic, the answer to Question \ref{quewhenperzar} can be either positive or negative.
Indeed, in Section \ref{subsectionnonhyp}, we give examples to show that for cohomologically non-hyperbolic maps, one can not determine whether the set of periodic points are Zariski dense from their dynamical degrees.

\subsubsection{Historical notes}
The first fundamental result for Question \ref{quewhenperzar} is the 
positive answer for polarized endomorphisms\footnote{Recall that an endomorphism $f$ on a projective variety $X$ is said to be {\em polarized} if there exists an ample line bundle $L$ on $X$ satisfying $f^*L=qL$ for some integer $q\geq 2$.  In this case, $\la_i(f)=q^i$ for $i=0,\dots, d$. Hence polarized endomorphisms are cohomologically hyperbolic.},
which can be achieved both by analytic and algebraic method. 

\medskip

Suppose that $X$ is smooth projective over $\C$ and $f$ is a  polarized endomorphism.
Using complex analytic methods,
 Briend-Duval~\cite{Briend2001} and subsequently Dinh-Sibony
\cite{Dinh2010} have proved that the set of periodic points is Zariski dense in $X$. By the Lefschetz principle, these results hold true whenever $\bk$ has characteristic zero.
Later, Hrushovski and Fakhruddin \cite{fa} gave a purely algebraic proof of the Zariski density of periodic points over any algebraically closed field\footnote{Their proof indeed works for amplified endomorphisms which are more general than polarized ones.}.
%We shall return to their approach later in this introduction.

\medskip

The complex analytic methods alluded to above have been used to give a positive answer to Question~\ref{quewhenperzar} for several other classes of maps. For example, building on the work of Guedj \cite{Guedj2005},  Dinh, Nguy\^en and Truong \cite{Tien-CuongDinha} proved it when $f$ is $(\dim X)$-cohomologically hyperbolic.
See \cite{Tien-CuongDinh,Bedford2005,Dujardin2006,Diller2010,Bedford1993,Jonsson2018}  for other cases obtained using complex analytic method.
All these cases are cohomologically hyperbolic, hence implied by our Theorem \ref{thmperiodisodenseint}.
However, when the complex analytic method works, usually it not only proves the Zariski density of periodic points, but also show that the periodic points equidistribute to the maximal entropy measure in the complex topology.

%On the other hand, Dinh and Sibony \cite{Tien-CuongDinh} proved the Zariski density of periodic points for automorphisms $f : \mathbb{A}^d \to \mathbb{A}^d$ on complex affine spaces that are regular (i.e. the indeterminacy loci of $f$ and its inverse $f^{-1}$ are disjoint on the hyperplane at infinity in $\mathbb{P}^m$). This result extends former works by Bedford-Lyubich-Smillie~\cite{Bedford1993} concerning H\'enon maps in dimension $2$. In \cite{Bedford2005,Dujardin2006,Diller2010}, Bedford, Diller, Dujardin and Guedj answered Question \ref{quewhenperzar} affirmatively for birational polynomial maps, or more generally for any birational self-maps on surfaces with $\la_1(f)>1$ satisfying the Bedford-Diller energy condition (which basically means that  the points of indeterminacy of $f^{-1}$ do not cluster too much near the points of indeterminacy of $f$).  Later in \cite{Jonsson2018}, Jonsson and Reschke proved that the Bedford-Diller energy condition is satisfied for all birational self-maps on surfaces with $\la_1(f)>1$ defined over a number field. 

%All the above cases are cohomologically hyperbolic, hence implied by our Theorem \ref{thmperiodisodenseint}.
%However, when the complex analytic method works, usually it not only proves the Zariski density of periodic points, but also show that the periodic points equidistribute to the maximal entropy measure in the complex topology.

\medskip

In \cite[Theorem 1.1]{Xie2015}, the author classified the birational self-maps on surfaces whose periodic points are not Zariski dense by algebraic method. The essential step of  \cite[Theorem 1.1]{Xie2015}, is the case where $\la_1(f)>1$ (hence cohomologically hyperbolic). An advantage of  the algebraic method is that it works in arbitrary characteristic. 

\subsubsection{Strategy of the proof}
We follow the original method of Hrushovski and Fakhruddin \cite{fa} by reducing our result to the case of finite fields. This method was also used in the proof of \cite[Theorem 1.1]{Xie2015}.

For the sake of simplicity, we shall assume that 
$X=\mathbb{P}^N$ and $f=[f_0:\dots: f_N]$ is a rational self-map having integral coefficients. Assume that $f$ is cohomologically hyperbolic. By Corollary \ref{corlscdynadegoverzin}, 
we can find a prime $p\geq 2$ such that the reduction $f_p$ modulo $p$ of $f$ is cohomologically hyperbolic. Then the set of periodic closed points of $f_p$ is Zariski dense in $\mathbb{P}^N(\overline{\mathbb{F}_p})$ by an argument of Fakhruddin based on  Hrushovski's twisted Lang-Weil estimate c.f. Theorem \ref{hrushovski}. 
We show that most of the $f_p$-periodic points are ``isolated" in a certain sense (c.f. Corollary \ref{corisoperiodiccohhyp}). The main ingredient of this step is a recursive inequality proved in Theorem \ref{thmmaincompudegone} and \cite[Proposition 3.5]{Matsuzawa}.
By Lemma \ref{dvr} (which generalizes \cite[Theorem 5.1]{fa}), we can lift isolated periodic points from the special fiber to the generic fiber. This concludes the proof.

\subsection{Kawaguchi-Silverman conjecture}\label{subintksc}
Assume that $\bk=\overline{\Q}.$ Let $X$ be a projective variety over $\overline{\Q}$ and let $L$ be an ample line bundle on $X$.
We denote by $h_L: X(\overline{\Q})\to \R$ a Weil height associated to $L$ c.f. \cite{Hindry2000,Bombieri2006}. It is unique up to adding a bounded function. 
Set $h_L^+:=\min\{1, h_L\}.$

Let $X_{f}(\overline{\Q})$ be the set of points $x\in X(\overline{\Q})$ whose orbit is well-defined i.e. $f^n(x)\not\in I(f)$ for every $n\geq 0.$ 
In \cite{Kawaguchi2016}, Kawaguchi and Silverman introduced the fundamental notion of arithmetic degree to describe the arithmetic complexity of an orbit.
For $x\in X(\overline{\Q})$, the \emph{upper/ lower arithmetic degree} for $X,f,x$ are
$$\overline{\alpha}_f(x):=\limsup_{n\to \infty}h_L^+(f^n)^{1/n} \text{ and } \underline{\alpha}_f(x):=\liminf_{n\to \infty}h_L^+(f^n)^{1/n}.$$
If $\overline{\alpha}_f(x)=\underline{\alpha}_f(x)$,
we set
\[
\alpha_f(x):=\overline{\alpha}_f(x)=\underline{\alpha}_f(x).
\]
In this case, we say that $\alpha_f(x)$ is well-defined and call it the \textit{arithmetic degree} of $f$ at $x$.

The following conjecture was proposed by Kawaguchi and Silverman \cite{Silverman2014,Kawaguchi2016}. It connects the arithmetic degree with the first dynamical degree.
\begin{con}[Kawaguchi-Silverman conjecture]\label{kscin}
Let $X$ be a projective variety over $\overline{\Q}$.
Let $f: X\dashrightarrow X$ be a dominant rational self-map. Then for every  $x\in X_f(\overline{\Q})$, $\alpha_f(x)$ is well defined.
Moreover, if $O_f(x)$ is Zariski dense, then we have $\alpha_f(x)= \la_1(f).$
\end{con}

The general form of the Kawaguchi-Silverman conjecture is wildly open. However many special cases are known especially when $f$ is well-defined everywhere. 
When $f$ is a polarized endomorphism, the Kawaguchi-Silverman conjecture is implied by the Northcott property. 
It was completely solved when $X$ is a projective surface and $f$ is a surjective endomorphism by Kawaguchi, Silverman, Matsuzawa, Sano, Shibata, Meng and Zhang \cite{Kawaguchi2008,Kawaguchi2014,Matsuzawa2018,Meng2022}. Except Kawaguchi's automorphism case, the proof heavily relies on classification (or minimal model theory) of surfaces. In higher dimension, serval cases are proven by minimal model theory. These results are on surjective endomorphisms on projective varieties (c.f. \cite{Meng2023}). Few results are known when $f$ is merely rational. The following are two remarkable cases. 
\begin{points}
\item[(1)]Conjecture \ref{kscin} was proved for regular affine automorphisms on $\A^N$ by Kawaguchi \cite{Kawaguchi2006,Kawaguchi2013}. 
\item[(2)]Conjecture \ref{kscin} holds by Matsuzawa and Wang \cite{Wang2022,Matsuzawa} when $X$ is a smooth projective variety, $f$ is a $1$-cohomologically hyperbolic rational map, and the $f$-orbit of $x$ is generic i.e. 
for every proper subvariety $Y$ of $X$, the set $\{n\geq 0|\,\, f^n(x)\in Y\}$ is finite.
\end{points}
The Dynamical Mordell-Lang conjecture proposed by Ghioca and Tucker asserts that for every $x\in X_{f}(\overline{\Q})$, if $O_f(x)$ is Zariski dense, then the $f$-orbit of $x$ is \emph{generic} c.f. \cite{Ghioca2009} (see also \cite[Conjecture 1.2]{Xie2023a}). So (2) implies that the Kawaguchi-Silverman conjecture for $1$-cohomologically hyperbolic self-maps assuming the Dynamical Mordell-Lang conjecture.
For more results, see \cite{Matsuzawa2023} and the references therein.

\medskip

In Section \ref{sectionksc}, we prove the following result (the case $\la_2(f)=\la_1(f)^2$ was already proved by Wang and Matsuzawa \cite[Theorem 1.17]{Matsuzawa}).
\begin{thm}\label{KSCsurfacelaonelain}Let $X$ be a projective surface over $\overline{\Q}$ and $f: X\dashrightarrow X$ be a dominant rational self-map such that $\la_1(f)>\la_2(f)$ or $\la_2(f)=\la_1(f)^2$. Let $x\in X_{f}(\overline{\Q})$. If the orbit $O_f(x)$ of $x$ is Zariski dense, then $\alpha_f(x)=\la_1(f).$
\end{thm}
In particular, Theorem \ref{KSCsurfacelaonelain} implies the Kawaguchi-Silverman conjecture for birational self-maps on projective surfaces. We first explain how to prove Theorem \ref{KSCsurfacelaonelain} assuming the Dynamical Mordell-Lang conjecture (which was done in (2)).
In the proof of (2), Matsuzawa and Wang construct some recursive inequality for the heights $h(f^n(x))$ when $f^n(x)$ is not contained in the base locus $B$ of some big line bundle. Applying the Dynamical Mordell-Lang conjecture, one can show that the orbit meets $B$ in at most finitely many times. 
So we may ignore the base locus and assume that the recursive inequality holds for all $n$. This implies our result easily. As the Dynamical Mordell-Lang conjecture is wildly open in general, we can not ignore the base locus.  This is the main difficulty of our proof. Our idea is to construct a weaker recursive inequality when $f^n(x)$ is contained in $B$. 
We then combine this inequality with the one when $f^n(x)\in B$ to get a lower bound of the growth of the height. We apply the Weak dynamical Mordell-Lang \cite[Corollary 1.5] {Bell2015} (see also \cite[Theorem 2.5.8]{Favre2000a}, \cite[Theorem D, Theorem E]{Gignac2014},\cite[Theorem 2]{Petsche2015}, \cite[Theorem 1.10]{Bell2020}, \cite[Theorem 1.17]{Xie2023} and \cite[Theorem 5.2]{Xie2023a}) to show that the density of $n$ with $f^n(x)\in B$ is zero. Using this, one can show that we can ignore $B$ asymptotically and conclude the proof.

\subsection{Further problems}
Though our Theorem \ref{thmcomputela} gives an algorithm to compute the dynamical degrees to any given precision, it seems that our algorithm is not so efficient.
Either theoretically, or practically, it is interesting to have a more efficient algorithm. 
In a private communication with Silverman, he asked the following more precise question:
\begin{que}\label{quesilvereff}
Is there an algorithm to compute the dynamical degrees $\la_i(f)$ to within $1/2^l$ using only $O(l^e)$ storage for some ``not too large" $e$? 
\end{que}

Blanc, Cantat and McMullen \cite{McMullen2007,Cantat2007,Blanc-Cantat} showed that there is a gap on the first dynamical degree for surface birational self-maps. More precisely, 
 for every surface birational self-map $f$, we have $\la_1(f)\not\in (1,\la_L)$, where  $\la_L\simeq 1.176280$ is the Lehmer number i.e. 
 the unique root $>1$ of the irreducible polynomial $x^{10}+ x^9-x^7-x^6-x^5-x^4-x^3+x+1$.
This result relies on the existence of algebraically stable models for surface birational self-maps
proved by Diller-Favre \cite{favre}. It is interesting to ask whether such a gap exists for general rational self-maps and for higher dynamical degrees. 

\begin{que}\label{quegapdy}Is there a $\la>1$ depending on $d\geq 1$ and $i\in \{1,\dots, d\}$ such that for every dominant rational self-maps $f$ on a $d$-dimensional projective variety $X$, we have $\la_i(f)\not\in(1,\la)$?
\end{que}
When $i=d$, Question \ref{quegapdy} has positive answer by taking $\la=2.$
Corollary \ref{corwellorded} gives positive answer to Question \ref{quegapdy} for self-maps coming from a given family. 
In particular, Corollary \ref{corwellorded}  shows that for every $d\geq 1$, $D\geq 1$, there is $\la>1$ depending on $d$ and $D$, such that  
for every $i=0,\dots, d$ and every dominant rational self-map $f$ of $\P^d$ with $\deg_{1, O(1)}f\leq D$, we have $\la_i\not\in(1,\la).$

The same question can be asked for the cohomological Lyapunov exponents.
\begin{que}\label{quegaply}Is there a $\mu>1$ depending on $d\geq 1$ and $i\in 1,\dots, d$ such that for every dominant rational self-maps $f$ on a $d$-dimensional projective variety $X$, we have $\mu_i(f)\not\in (\mu^{-1},1)\cup (1,\mu).$
\end{que}

%
%\subsection{Organization}
%The article is organized in 8 sections. In Section \ref{sectionbirmod}, we introduce the terminology of cocycles on birational models. 
%In Section \ref{secrecurdeg}, we construct a family of recursive inequalities of some dynamically meaningful cycles via computation of mixed degrees. This is the main tool in the whole paper. We also use them to get the recursive inequalities for the dynamical degrees. In Section \ref{secungrowdeg}, we prove Corollary \ref{cordegreestableint} as an application of the  recursive inequalities developed in Section \ref{secrecurdeg}. We give an example to show that Corollary \ref{cordegreestableint}  does not hold for general submultiplicative sequences. In Section \ref{sectionalgdeg}, we give the first algorithm to compute the dynamical degrees to any given precision c.f. Theorem \ref{thmcomputela}.
%We also give better lower bound in the surface case c.f. Theorem \ref{thmsurfaceboundin}. In Section \ref{sectionsslsc}, we prove the lower semi-continuity of the dynamical degrees c.f. Theorem \ref{thmlscdynadegin}. In Section \ref{sectioncohhyp}, we prove Theorem \ref{thmperiodisodense} saying that the periodic points of a cohomologically hyperbolic self-map are Zariski dense. In Section \ref{sectionksc}, we prove the Kawaguchi-Silverman conjecture for self-maps on projective surfaces with $\la_1(f)>\la_2(f)$ or $\la_1(f)^2=\la_2(f)$ c.f. 
%Theorem \ref{KSCsurfacelaonelain}.

\subsection*{Acknowledgement}
The author would like to thank Long Wang for interesting discussion on the Kawaguchi-Silverman conjecture.
The author would like to thank Zhiqiang Li and Xianghui Shi for interesting discussion on computability theory. 
The author would like to thank Tien Cuong Dinh, Keiji Oguiso, Mattias Jonsson, Claude-Michel Viallet, Guolei Zhong and Fei Hu for helpful comments on the first version of this work.
The author would like to thank Joseph Silverman for proposing Question \ref{quesilvereff}.
The author would like to thank Serge Cantat for his careful reading of first version and his helpful comments.
The author would like to thank Matsuzawa, who found a mathematical mistake in the first version of the paper.

\section{Birational models}\label{sectionbirmod}
In this section,  we introduce the terminology of cocycles on birational models. 
This terminology is not completely necessary for our paper, but it naturally fits our setting and it simplifies the notations a lot and makes the presentations clearer.

	Let $\bk$ be a field. Let $X$ be a projective variety of dimension $d$ over $\bk.$
	A \emph{birational model} of $X$ is a projective variety $X_{\pi}$ with a birational morphism $\pi : X_\pi \rightarrow X $.
	For two birational models $X_\pi$ and $X_{\pi'}$, we say that \emph{$X_{\pi'}$ dominates $X_{\pi}$} and write $X_{\pi'}\geq X_{\pi}$ if the birational map $\mu:=\pi^{-1}\circ \pi':X_{\pi'}\rightarrow X_{\pi}$ is a morphism.

\subsection{Line bundles}\label{subseclinebun}
Let $\widetilde{\Pic}(X)$ and $\widetilde{\Pic}(X)_{\R}$ be the inductive limits $$\widetilde{\Pic}(X):=\varinjlim_{\pi}\Pic(X_{\pi})$$ 
and $$\widetilde{\Pic}(X)_{\R}:=\varinjlim_{\pi}\Pic(X_{\pi})_{\R}.$$ 
with respect to pullback arrows. 
In particular, $\widetilde{\Pic}(X)_{\R}=\widetilde{\Pic}(X)\otimes_{\Z}\R$.
To simplify the notations, for $L\in \widetilde{\Pic}(X)$, we still denote by $L$ its image in $\widetilde{\Pic}(X)_{\R}$.

\medskip

For every element $L\in \widetilde{\Pic}(X)$ (resp. $L\in \widetilde{\Pic}(X)_{\R}$) there is a birational model $X_{\pi}$ of $X$ such that $L$ is represented by some $L_{\pi}\in \Pic(X_{\pi})$ (resp. $L_{\pi}\in \Pic(X_{\pi})_{\R}$); we say that $L$ is defined on $X_{\pi}$ by $(\pi^{-1}\circ \pi')^*L_{\pi}.$
For every $X_{\pi'}\geq X_{\pi}$, $L$ is also defined on $X_{\pi'}.$
We say that $L\in \widetilde{\Pic}(X)_{\R}$ is big (resp. nef, effective, pseudo-effective) if it is represented by $L_{\pi}\in \Pic(X_{\pi})_{\R}$ for some birational model $X_{\pi}$ of $X$ such that $L_{\pi}$ is big (resp. nef, effective, pseudo-effective).  For $L,M\in \widetilde{\Pic}(X)_{\R}$, write $L>_{big}M$ if $L-M$ is big and $L\geq M$ if $L-M$ is pseudo-effective.

\medskip

For $L\in \widetilde{\Pic}(X)$ and every birational model $X_{\pi}$ of $X$, define the \emph{stable base locus} as follows:
Pick any model $\pi_0: X_{\pi_0}\to X$ such that $L$ is defined on $X_{\pi_0}$,  $X_{\pi_0}$ is normal and $X_{\pi_0}$ dominates $X_{\pi}.$
Define $\bB_{X_{\pi_0}}(L):=\cap_{n\geq 0}Bs_{X_{\pi_0}}(nL)$ where $Bs_{X_{\pi_0}}(\cdot)$ is the usual base locus, and let $\bB_{X_{\pi}}(L)$ be the image of $\bB_{X_{\pi_0}}(L)$ in $X_{\pi}.$
It is easy to check that this definition does not depend on the choice of $\pi_0.$ If $L$ is effective, then $\bB_{X_{\pi}}(L)\neq X_{\pi}.$

\medskip

Let $C$ be a curve in $X$ and $L\in \Pic(X)_{\R}$. Assume that $L$ is represented by some $L_{\pi}\in \Pic(X_{\pi})_{\R}$ on some birational model $X_{\pi}$ of $X$ such that $C$ is not contained in the indeterminacy locus $I(\pi^{-1})$ of $\pi^{-1}: X\dashrightarrow X_{\pi}.$ Let $C_{\pi}:=\overline{\pi^{-1}(C\setminus I(\pi^{-1}))}$ be the strict transform of $C$ by $\pi.$
We define $(C\cdot L)$ to be $(C_{\pi}\cdot L_{\pi})$. It is easy to see that this definition does depend on the choice of $X_{\pi}, L_{\pi}.$

\medskip

Let $f: X\dashrightarrow X$ be a rational self-map and $C$ be a curve in $X$.  Assume that $C$ is not contained in $I(f)$. Define $f_*(C)$ as follows:
Pick a birational model $X_{\pi}$ of $X$ such that $f_{\pi}:=f\circ \pi: X_{\pi}\to X$ is a morphism and $C\not\subseteq I(\pi^{-1})$. Such a model exists, as we can pick $X_{\pi}$ to be the graph of $f$ in $X\times X$ and let $\pi$ be the projection to the first factor. Define $f_*(C):=(f_{\pi})_*(C_{\pi})$. This definition does not depend on the choice of $X_{\pi}$.
Let $M\in \Pic(X)_{\R}$. We pick $X_{\pi}$ as above. Then $M$ is defined on $X_{\pi}.$
The previous paragraph shows that the intersection $(f^*(M)\cdot C)$ is well-defined and is equal to $(C_{\pi}\cdot f_{\pi}^*M)$.
By projection formula, we get 
\begin{equation}\label{equationprofubd}(f^*(M)\cdot C)=(M\cdot f_*(C)).
\end{equation}

\subsection{Cocycles}\label{subsectioncocycles}
For $i=0,\dots, d$, let $\widetilde{\CH^i}(X)_{\R}$ be the inductive limit$$\widetilde{\CH^i}(X)_{\R}:=\varinjlim_{\pi}\CH^i(X_{\pi})_{\R}$$ 
with respect to pullback arrows, where $\CH^i(\cdot)$ is the Chow group of degree $i$ cocycles \cite{Fulton1984}.
In particular, $\widetilde{\CH^1}(X)_{\R}=\widetilde{\Pic}(X)_{\R}$.

For $L_1,\dots, L_i\in \widetilde{\Pic}(X)_{\R}$ and $Z\in \widetilde{\CH^j}(X)_{\R}$, we define the intersection $L_1\cdots L_i\cdot Z\in \widetilde{\CH^{i+j}}(X)_{\R}$ as follows: There is a birational model $X_{\pi}$
of $X$ such that $L_1,\dots,L_i$ and $Z$ are all defined on $X_{\pi}$. Let $L_{1,\pi},\dots, L_{i, \pi}\in \Pic(X_{\pi})$ and $Z_{\pi}\in \CH^j(X_{\pi})_{\R}$ represent $L_1,\dots, L_i$ and $Z$. Define $L_1\cdots L_i\cdot Z$ to be the element in $\widetilde{\CH^{i+j}}(X)_{\R}$ represented by $L_{1,\pi}\cdots L_{i, \pi}\cdot Z_{\pi}\in \CH^{i+j}(X_{\pi}).$ This definition does not depend on the choice of the birational model $X_{\pi}$.

For $P\in \widetilde{\CH^d}(X)_{\R}$, define $(P)$ to be the degree of $P_{\pi}$ where $P_{\pi}\in \CH^d(X_{\pi})$ for some birational model $X_{\pi}$ of $X$ which defines $P.$ This does not depend on the choice of  birational model $X_{\pi}.$

\medskip

Let $g: Y\dashrightarrow X$ be a rational map.  We define the pullback $g^*: \widetilde{\CH^i}(X)_{\R}\to \widetilde{\CH^i}(Y)_{\R}$ as follows: 
For every $Z\in \widetilde{\CH^i}(X)_{\R}$, there is a birational model $X_{\pi}$ of $X$ and $Z_{\pi}\in \CH^i(X_{\R})$ such that  $Z_{\pi}$ defines $Z$.
There is a  birational model $Y_{\phi}$ of $Y$ such that $g$ induces a morphism $g': Y_{\phi}\to X_{\pi}.$ We define $f^*Z$ to be the element in $\widetilde{\CH^i}(X)_{\R}$
defined by $(g')^*Z_{\pi}\in \CH^i(Y_{\phi})_{\R}.$ This definition does not depend on the choice of $X_{\pi}$ and $Y_{\phi}.$
\medskip

For $Z, W\in \widetilde{\CH^i}(X)_{\R}$, write $Z\geq_n W$ if for every dominant and generically finite rational map $g: Y\dashrightarrow X$ and 
$(d-i)$-tuple of nef line bundles $H_1,\dots, H_{d-i}$ in $\widetilde{\Pic}(Y)_{\R},$
$(g^*(Z-W)\cdot H_1\cdots H_{d-i})\geq 0.$ Write $Z>_{n} W$ if for some big and nef line bundle $L\in \widetilde{\Pic}(X)_{\R}$, $Z\geq_n W+L^i.$
\rem Note that $Z\geq_n W$ with $Z\not>_n W$ does not imply $Z=W.$
\endrem

For $Z, W\in \widetilde{\CH^i}(X)_{\R}$, with $Z\geq_n W$ (resp. $Z>_n W$), and a dominant and generically finite rational map $g: Y\dashrightarrow X$, 
we have $g^*Z\geq_n W$, (resp. $g^*Z>_n g^*W$).

\subsection{Siu's inequalities}
Siu's inequality \cite[Theorem 2.2.13]{Lazarsfeld} for nef line bundles is useful in our paper. For the convenience of the applications, we write it in the following form for nef line bundles in $\widetilde{\Pic}(X)_\R.$
\begin{thm}\label{thmsiuincodimo} Let $L,M$ be nef line bundles in $\widetilde{\Pic}(X)_\R.$ Assume that $(M^d)>0,$ then
 $$L\leq d\frac{(L\cdot M^{d-1})}{(M^d)}M.$$
In particular, for every $\epsilon\in (0,1)$, $\epsilon L<_{big} d\frac{(L\cdot M^{d-1})}{(M^d)}M .$
\end{thm}

Applying Siu's inequality inductively, Dang proved a version of Siu's inequality in arbitrary codimension \cite[Corollary 3.4.6]{Dang2020}. 
In \cite[Theorem 3.5]{Jiang2023}, Jiang and Li give another proof using (multipoint) Okounkov bodies and get the optimal coefficient.
For the convenience of the applications, we write \cite[Theorem 3.5]{Jiang2023} in the following form for nef line bundles in $\widetilde{\Pic}(X).$
\begin{thm}\label{thmsiuin}Let $i=0,\dots,d,$ and $L_1,\dots,L_i,M$ be nef line bundles in $\widetilde{\Pic}(X)_\R.$ 
Then we have $$(M^d)L_1\cdots L_i\leq_n \binom{d}{i}(L_1\cdots L_i\cdot M^{d-i})M^i$$
i.e. 
 for every $(d-i)$-tube of nef line bundles $H_1,\dots, H_{d-i}$ in $\widetilde{\Pic}(X)_{\R},$
we have $$(L_1\cdots L_i\cdot H_1\cdots H_{d-i})(M^d)\leq \binom{d}{i}(L_1\cdots L_i\cdot M^{d-i})(M^i\cdot H_1\cdots H_{d-i}).$$ 

In particular, for every $\epsilon\in (0,1)$, 
$$\epsilon(M^d)L_1\cdots L_i<_n \binom{d}{i}(L_1\cdots L_i\cdot M^{d-i})M^i.$$
\end{thm}

\section{Recursive inequalities for degree sequences}\label{secrecurdeg}
	Let $\bk$ be a field. Let $X$ be a projective variety of dimension $d$ over $\bk.$ Let $f: X\dashrightarrow X$ be a dominant rational self-map. Let $L$ be a big and nef line bundle in $\widetilde{\Pic}(X)$. 
%	Recall that, for $i=0,\dots, d$, the $i$-th degree of $f$ is $$\deg_{i,L}f:=((f^*L)^i\cdot L^{d-i}).$$
%
%The $i$-th dynamical degree of $f$ is
%$$
%	\la_i(f):=\lim_{n\to\infty}(\deg_{i,L} f^n)^{1/n}\geq 1.
%$$
%The limit converges and does not depend on the choice of $L$
%\cite{Russakovskii1997, Dinh2005, Truong2020,Dang2020}.
%
%\medskip
%
%
%For $i=1,\dots, d$, recall that the $i$-th cohomological Lyapunov exponent of $f$ is $$\mu_i(f):=\la_i(f)/\la_{i-1}(f).$$
%Define $\mu_{d+1}(f):=0$ for convenience.
%As the sequence of dynamical degrees is log-concave \cite{Dinh2005, Truong2020,Dang2020}, the sequence $\mu_i(f), i=1\dots, d$ is decreasing. 
%\medskip

To simplify the notations, we write $$\la_i:=\la_i(f), \mu_i:=\mu_i(f), \deg_if^n:= \deg_{i,L}f^n, \text{ and } L_n:=(f^n)^*L$$ for every $n\geq 0.$

\subsection{A lemma on recursive inequalities}	
The following simple lemma on recursive inequalities is useful.
	\begin{lem}\label{lemrecuineq}Let $a_n, n\geq 0$ be a sequences of non-negative real numbers.  
Let $\alpha, \beta, \gamma$ be real numbers with $\alpha\geq 0$ and $\gamma\geq \alpha+\beta$. Assume that $a_1>\beta a_0$ and 
$$a_{n+2}+\alpha\beta a_n\geq \gamma a_{n+1}$$
for every $n\in \{0,\dots, N\}$ where $N\in \Z_{\geq 0}\cup \{+\infty\}$.
Then for every $n\in \{0,\dots, N\}$, we have $$a_{n+2}>\beta a_{n+1} \text{ and } a_{n+2}\geq \alpha^n(a_1-\beta a_0).$$
In particular, if $N=+\infty$, then $\liminf\limits_{n\to \infty}a_n^{1/n}\geq \alpha.$
\end{lem}
\proof
As $\gamma\geq \alpha+\beta$ and $a_{n+1}\geq 0$, we have 
$a_{n+2}+\alpha\beta a_n\geq (\alpha+\beta)a_{n+1}.$
Hence we have $(a_{n+2}-\beta a_{n+1})\geq \alpha(a_{n+1}-\beta a_n),$ which concludes the proof.
\endproof

\subsection{Mixed degrees}
Let $s\geq 1$, consider two sequence of non-negative integers: $m_1> \dots >m_s\geq 0$ and $r_1,\dots, r_s\geq 0$ with $\sum_{i=1}^sr_i=d.$
We will compute the \emph{mixed degree}, which is defined to be $$(L_{m_1}^{r_1}\cdots L_{m_s}^{r_s}).$$
 Our computation is based on a direct application of the higher codimensional Siu's inequality.

\medskip

\begin{lem}\label{lemsiucomp}Let $r_1,r_2\geq 0$ with $r_1+r_2\leq d$.
Let $A$ be a product of $d-r_1-r_2$ nef line bundles in $\widetilde{\Pic}(X).$
%and $A\in \widetilde{\rm BPF}^{d-r_1-r_2}(X)$. 
 Let $n_1,n_2\geq 0$ and $0\leq t\leq r_1$.
If $n_1\geq n_2$, then  we have $$(L_{n_1}^{r_1}\cdot L_{n_2}^{r_2}\cdot A)\leq \binom{d-r_1-r_2+t}{t}\frac{\deg_{r_1}f^{n_1-n_2}}{\deg_{r_1-t}f^{n_1-n_2}}(L_{n_1}^{r_1-t}\cdot L_{n_2}^{r_2+t}\cdot A);$$
if $n_2\geq n_1$, then we have  $$(L_{n_1}^{r_1}\cdot L_{n_2}^{r_2}\cdot A)\leq \binom{d-r_1-r_2+t}{t}\frac{\deg_{d-r_1}f^{n_2-n_1}}{\deg_{d-r_1+t}f^{n_2-n_1}}(L_{n_1}^{r_1-t}\cdot L_{n_2}^{r_2+t}\cdot A).$$
\end{lem}
\proof
Up to some small pertubations of $L$ of the form $L+\epsilon H$ for some $H$ ample and positive rational $\epsilon\to 0$, we can
suppose that $L$ is ample on $X.$ After replacing $L$ by some positive multiple, we may further assume that $L$ is very ample.

\medskip

Replace $X$ by a sufficiently large model, we may assume that $L_{n_1}$ and $L_{n_2}$ are defined over $X$. Hence they are generated by global sections.
Let $V$ be the intersection of $(r_1-t)$ general sections of $L_{n_1}$ and $r_2$ general sections of $L_{n_2}$.
Then $V=L_{n_1}^{r_1-t}\cdot L_{n_2}^{r_2}$ in $\CH^{r_1+r_2-t}(X)$.

By Theorem \ref{thmsiuin}, we have $$L_{n_1}^t|_{V}\leq \binom{d-r_1-r_2+t}{t}\frac{((L_{n_1}|_V)^s\cdot (L_{n_2}|_V)^{d-r_1-r_2})}{((L_{n_2}|_V)^{d-r_1-r_2+t})} L_{n_2}^{t}|_V.$$
Hence
$$L_{n_1}^t\cdot V\leq \binom{d-r_1-r_2+t}{t}\frac{(L_{n_1}^t\cdot L_{n_2}^{d-r_1-r_2}\cdot V)}{(L_{n_2}^{d-r_1-r_2+t}\cdot V)}L_{n_1}^{r_1-t}\cdot L_{n_2}^{r_2+t}.$$
Intersecting with $A$, we conclude the proof by the projection formula. 
\endproof

\medskip

Applying Lemma \ref{lemsiucomp}, we get upper and lower bounds on the mixed degrees. 

	\begin{pro}\label{proupperboundmixdeg}Let $l_i:=r_1+\dots+r_i$, we have
$$(L_{m_1}^{r_1}\cdots L_{m_s}^{r_s})\leq (L^d)\prod_{i=1}^s\binom{d-r_{i+1}}{l_i}\prod_{i=1}^s\frac{\deg_{l_i,L}(f^{m_i-m_{i+1}})}{(L^d)}$$
and 
$$\deg_d(f^{m_1})\prod_{i=1}^{s-1}\binom{d-l_i}{r_{i+1}}^{-1}\prod_{i=1}^{s-1}\frac{\deg_{l_i}(f^{m_1-m_{i+1}})}{\deg_{l_{i+1}}(f^{m_1-m_{i+1}})}\leq (L_{m_1}^{r_1}\cdots L_{m_s}^{r_s})$$
\end{pro}	
\proof
Apply the first inequality in Lemma \ref{lemsiucomp} for $L_{m_1}^{r_1}$, $L_{m_2}^{r_2}$, $t=r_1$ and $A:=L_{m_2}^{r_2}\cdots L_{m_s}^{r_s}$, 
we have
\begin{flalign*}
(L_{m_1}^{l_1}\cdots L_{m_s}^{l_s-l_{s-1}})\leq \binom{d-r_{2}}{l_1}\frac{\deg_{l_1}(f^{m_1-m_2})}{(L^d)}(L_{m_2}^{l_2}\cdots L_{m_s}^{l_s-l_{s-1}}).
\end{flalign*}
We get the first inequality by induction.

\medskip

Apply the first inequality in  Lemma \ref{lemsiucomp} for $L_{m_1}^{l_2}$, $L_{m_2}^{0}$, $t=r_2$ and $A:=L_{m_3}^{l_3-l_2}\cdots L_{m_s}^{l_s-l_{s-1}}$,
we have
\begin{flalign*}
(L_{m_1}^{l_2}\cdot L_{m_3}^{l_3-l_2}\cdots L_{m_s}^{l_s-l_{s-1}})\leq \binom{d-l_1}{r_2}\frac{\deg_{l_2}(f^{m_1-m_2})}{\deg_{l_1}(f^{m_1-m_2})}(L_{m_1}^{l_1}\cdots L_{m_s}^{l_s-l_{s-1}}).
\end{flalign*}
Hence we have
\begin{flalign*}
(L_{m_1}^{l_1}\cdots L_{m_s}^{l_s-l_{s-1}})\geq\binom{d-l_1}{r_2}^{-1}\frac{\deg_{l_1}(f^{m_1-m_2})}{\deg_{l_2}(f^{m_1-m_2})}(L_{m_1}^{l_2}\cdot L_{m_3}^{l_3-l_2}\cdots L_{m_s}^{l_s-l_{s-1}}).
\end{flalign*}
We get the second inequality by induction, which concludes the proof.
\endproof	

\medskip

By Proposition \ref{proupperboundmixdeg}, we get the following corollary directly.

\begin{cor}\label{cormixddyna}
For every $\delta\in (0,1)$, there is a constant $D_{\delta}\geq 1$ such that 
$$D_{\delta}^{-1}\delta^{m_1}\prod_{i=1}^s\la_{l_i}(f)^{m_i-m_{i+1}}\leq (L_{m_1}^{r_1}\cdots L_{m_{s}}^{r_s})\leq D_{\delta}\delta^{-m_1}\prod_{i=1}^s\la_{l_i}(f)^{m_i-m_{i+1}}.$$
where $l_i:=r_1+\dots+r_i.$
\end{cor}

\subsection{Recursive inequalities}\label{subsectionrecrel}
For two functions $\theta_1,\theta_2: \Z_{\geq 0}\to \R_{>0}$, define $\theta_1\gtrsim \theta_2$ if 
$$\limsup\limits_{n\to \infty} (\theta_2/\theta_1)^{1/n}\leq 1.$$
This defines a partial ordering on the space of functions from $\Z_{\geq 0}$ to $\R_{>0}.$
Define $\theta_1\approx \theta_2$ if $\theta_1\gtrsim \theta_2$ and $\theta_2\gtrsim \theta_1$. This is an equivalence relation.

\medskip

\begin{thm}\label{thmmaincompudegany}For $r_1,r_2\geq 0$ with $r_1+r_2+1\leq d$, set $$t:=\min\{i\geq 1|\,\ \mu_{r_1+1+i}\mu_{r_1+r_2+1+i}<\mu_{r_1+1}\mu_{r_1+r_2+1}\}.$$
%As $\mu_i, i=1,\dots, d+1$ is decreasing, one of $\mu_{r_1+t}\mu_{r_1+r_2+1+t}, \mu_{r_1+t+1}\mu_{r_1+r_2+t}$ is strictly less than $\mu_{r_1+1}\mu_{r_1+r_2+1}.$
%Set $$\eta_{r_1,r_2}:= \max(\{\mu_{r_1+t}\mu_{r_1+r_2+1+t}, \mu_{r_1+t+1}\mu_{r_1+r_2+t}\}\cap (0, \mu_{r_1+1}\mu_{r_1+r_2+1})).$$
Set $$\eta_{r_1,r_2}:= \mu_{r_1+t}\mu_{r_1+r_2+t+1}.$$
Then for every $\epsilon\in (0,1)$, there is $m_{\epsilon}>0$, such that for every $m\geq m_{\epsilon},$
$$\frac{(L_{2m}+\eta_{r_1,r_2}^mL)^{d-r_2}\cdot L_m^{r_2}}{\mu_{r_1+1}^{m}(L_{2m}+\eta_{r_1,r_2}^mL)^{d-r_2-1}\cdot L_m^{r_2+1}}>(d-r_2)\epsilon^m.$$
In particular, we have 
$$L_{2m}\cdot L_m^{r_2}+\eta_{r_1,r_2}^mL\cdot L_{m}^{r_2}>_{n}\epsilon^m\mu_{r_1+1}^{m}L_m^{r_2+1}.$$
\end{thm}

\begin{rem}\label{remstricineqeta}
We note that $\eta_{r_1,r_2}$ in Theorem \ref{thmmaincompudegone} is at most $\mu_{r_1+1}\mu_{r_1+r_2+1}$. Moreover, if $\mu_{r_1+1}=\mu_{r_1+r_2+1}$  or $\mu_{r_1+r_2+1}>\mu_{r_1+r_2+2}$, then 
$\eta_{r_1,r_2}$ is strictly leas than $\mu_{r_1+1}\mu_{r_1+r_2+1}$.
\end{rem}

\proof[Proof of Theorem \ref{thmmaincompudegany}]
Set $\eta:=\eta_{r_1,r_2}$.
We note that $\mu_{r_1+1+i}$ and $\mu_{r_1+r_2+1+i}$ are constant when $i\in[0,t-1].$
In particular $\mu_{r_1+1}=\mu_{r_1+t}.$

\medskip
 The decreasing of $\mu_i$ implies that $\eta^{(d-r_2-j)m}\la_j^m\la_{j+r_2}^m, j=0,\dots, d-r_2$ takes maximal value when $j=r_1+t.$
By  Corollary \ref{cormixddyna}, we have 
\begin{flalign}
(L_{2m}+\eta^mL)^{d-r_2}\cdot L_m^{r_2}\approx&\sum_{j=0}^{d-r_2}\eta^{(d-r_2-j)m}(L_{2m}^{j}\cdot L_{m}^{r_2} \cdot L^{d-j})\notag\\
\approx&\max_{j=0}^{d-r_2}\eta^{(d-r_2-j)m}\la_j^m\la_{j+r_2}^m\notag\\
\label{equeqleftd}\approx& \eta^{(d-r_2-r_1-t)m}\la_{r_1+t}^m\la_{r_1+r_2+t}^m
\end{flalign}

By Corollary \ref{cormixddyna}, we have
\begin{flalign*}
\mu_{r_1+1}^{m}(L_{2m}+\eta^mL)^{d-r_2-1}\cdot L_m^{r_2+1}\approx&\mu_{r_1+1}^{m}\sum_{j=0}^{d-r_2-1}\eta^{(d-r_2-1-j)m}(L_{2m}^{j}\cdot L_{m}^{r_2+1} \cdot L^{d-j})\notag\\
\approx&\mu_{r_1+1}^{m}\sum_{j=0}^{d-r_2-1}\eta^{(d-r_2-1-j)m}\la_j^m\la_{j+r_2+1}^m.
\end{flalign*}
The maximal taken when $j=r_1+t-1.$ Then we have 
\begin{flalign}
\mu_{r_1+1}^{m}(L_{2m}+\eta^mL)^{d-r_2-1}\cdot L_m^{r_2+1}\approx& \mu_{r_1+1}^{m}\eta^{(d-r_2-r_1-t)m}\la_{r_1+t-1}^m\la_{r_1+t+r_2}^m\notag\\
=& \mu_{r_1+t}^{m}\eta^{(d-r_2-r_1-t)m}\la_{r_1+t-1}^m\la_{r_1+t+r_2}^m\notag\\
\label{equartlm}=& \eta^{(d-r_2-r_1-t)m}\la_{r_1+t}^m\la_{r_1+t+r_2}^m
\end{flalign}
By (\ref{equeqleftd}) and (\ref{equartlm}), we conclude the proof by Theorem \ref{thmsiuincodimo}.
\endproof

\medskip

For every $i=1,\dots,d$, define $U(i):=\max\{j=0,\dots, d|\,\, \mu_j=\mu_i\}.$
If we take $r_1=i-1$ and $r_2=0$ in Theorem \ref{thmmaincompudegany}, we have $\eta_{r_1,r_2}=\mu_i\mu_{U(i)+1}$ and we get the following special case for line bundles.
 
\begin{thm}\label{thmmaincompudegone}For $i=1,\dots, d$ and every $\epsilon\in (0,1)$, there is $m_{\epsilon}>0$, such that for every $m\geq m_{\epsilon},$
$$L_{2m}+\mu_i^{m}\mu_{U(i)+1}^mL-\epsilon^m\mu_i^{m}L_m$$ is big.
\end{thm}

A weaker version of Theorem \ref{thmmaincompudegone} was proved in \cite[Proposition 3.5]{Matsuzawa}, when $\la_i=\max_{j=0}^d \la_j$.

\rem\label{remimprorecu}
As $\mu_i=\mu_{U(i)}$ and $\mu_{i+1}\geq \mu_{U(i)+1}$, Theorem \ref{thmmaincompudegone} can be reformulated in the following form:
For $i=1,\dots, d$ and every $\epsilon\in (0,1)$, there is $m_{\epsilon}>0$, such that for every $m\geq m_{\epsilon},$ 
$$L_{2m}+\mu_i^{m}\mu_{i+1}^mL-\epsilon^m\mu_i^{m}L_m$$ is big. 
\endrem

\medskip

We define three families of conditions on $X,f,L$ which depend only on some intersection numbers.
Let $\alpha_1,\dots, \alpha_d\in \R_{>0}$ be a sequence of decreasing numbers. Let $\gamma,\epsilon \in (0,1)$, and $m\in \Z_{\geq 1}.$
Set $\beta_i:=\prod_{j=1}^{i}\alpha_i$.

\begin{defi}For $i=1,\dots,d$, we say that $X,f,L$ has condition
$I_i(\alpha_1,\dots, \alpha_d; \gamma; \epsilon; m)$ if for every $j=0,\dots,i-1$, we have
$$\frac{(L_{2m}+\alpha_{j+1}^m\alpha_i^m\gamma^mL)^{d-i+j+1}\cdot L_m^{i-j-1}}{\alpha_{j+1}^{m}(L_{2m}+\alpha_{j+1}^m\alpha_i^m\gamma^mL)^{d-i+j}\cdot L_m^{i-j}}>(d-i+j+1)\epsilon^m.$$
\end{defi}

\medskip

\begin{defi}For $i=1,\dots,d$, define
$$B(\alpha_1,\dots, \alpha_d; \gamma; \epsilon; m):=\sum_{j=0}^{i-1}\binom{d}{j}\binom{d}{i-1}\frac{\alpha_i^m\gamma^m}{\epsilon^{m(j+1)}\beta_j^m}\frac{\deg_j(f^m)\deg_{i-1}(f^m)}{(L^d)^2}.$$ 
We say that $X,f,L$ has condition
$J_i(\alpha_1,\dots, \alpha_d; \gamma; \epsilon; m)$ if 
$$B(\alpha_1,\dots, \alpha_d; \gamma; \epsilon; m)<\epsilon^{2mi}\beta_i^m(1-\epsilon^{mi}).$$
\end{defi}

By Theorem \ref{thmsiuincodimo}, If $(X,f,L)$ has condition $I_i(\alpha_1,\dots, \alpha_d; \gamma; \epsilon; m)$, then 
$$L_{2m}\cdot L_m^{i-j-1}+\alpha_{j+1}^m\alpha_i^m\gamma^m L_{m}^{i-j-1}\cdot L>_{n}\epsilon^m \alpha_{j+1}^mL_{m}^{i-j}.$$

\begin{defi}
For $i=1,\dots,d$ and $N\geq 0$, we say that $X,f,L$ has condition $K_i(\alpha_1,\dots, \alpha_d; \gamma; \epsilon; m;N)$
if $$\deg_{i}(f^{m(N+1)})>B(\alpha_1,\dots, \alpha_d; \gamma; \epsilon; m)\epsilon^{-mi}\deg_{i}(f^{mN}).$$
\end{defi}

\medskip

Condition $I_i(\alpha_1,\dots, \alpha_d; \gamma; \epsilon; m)$ and $J_i(\alpha_1,\dots, \alpha_d; \gamma; \epsilon; m)$ only depend on the top intersection numbers using $L_{2m}, L_m, L$. 
Condition $K_i(\alpha_1,\dots, \alpha_d; \gamma; \epsilon; m;N)$ only depends on the top intersection numbers of $L_{Nm}, L_{(N+1)m}, L_{2m}, L_m, L$. 

\medskip

\begin{rem}\label{remopencd} If we fix $X,f,L,i,m$ and $n$, $I_i(\alpha_1,\dots, \alpha_d; \gamma, \epsilon, m), J_i(\alpha_1,\dots, \alpha_d; \gamma, \epsilon, m)$ and $K_i(\alpha_1,\dots, \alpha_d; \gamma, \epsilon, m; n)$ are open conditions on $(\alpha_1,\dots, \alpha_d, \gamma, \epsilon)$.
\end{rem}

\begin{lem}\label{lemcondiimpliesdyna}If $(X,f,L)$ has conditions $I_i(\alpha_1,\dots, \alpha_d; \gamma; \epsilon; m)$.
Then for every $n\geq 0$, we have 
$$\deg_i(f^{m(n+2)})+B\epsilon^{mi}\beta_i^m\deg_i(f^{mn})\geq \epsilon^{mi}\beta_i^m\deg_i(f^{m(n+1)}),$$
where $B:=B(\alpha_1,\dots, \alpha_d; \gamma; \epsilon; m).$
Assume further that the conditions $J_i(\alpha_1,\dots, \alpha_d; \gamma; \epsilon; m)$ and $K_i(\alpha_1,\dots, \alpha_d; \gamma; \epsilon; m;N)$ are  satisfied for some $N\geq 0$. Then 
for every $n\geq 1$, we have 
$$\deg_i (f^{m(N+n)})-\epsilon^{-mi}B\deg_{i}(f^{m(N+n-1)})\geq (\epsilon^{2mi}\beta_i^m)^{n-1}(\deg_{i}f^{m(N+1)}-\epsilon^{-mi}B\deg_{i}f^{mN}).$$
In particular, $K_i(\alpha_1,\dots, \alpha_d; \gamma; \epsilon; m;N+n)$ is satisfied for every $n\geq 0$ and 
we have 
$$\la_i\geq \epsilon^{2i}\beta_i.$$
\end{lem}
\proof
As $I_i(\alpha_1,\dots, \alpha_d; \gamma; \epsilon; m)$ is satisfied, for every $j=0,\dots,i-1$, we have 
\begin{equation}\label{equbigseq}L_{2m}^{j+1}\cdot L_m^{i-j-1}+\alpha_{j+1}^m\alpha_i^m\gamma^mL_{2m}^j\cdot L_{m}^{i-j-1}\cdot L>_{big} \epsilon^m \alpha_{j+1}^mL_{2m}^jL_{m}^{i-j}.\end{equation}

To simplify the notations, for $u,v,w\geq 0$ with $u+v+w\leq d$, write $$D^{u,v,w}(n):=(L_{(2+n)m}^{u}\cdot L_{m(n+1)}^{v}\cdot L_{mn}^{w}\cdot L^{d-u-v-w}).$$
Apply $(f^n)^*$ to (\ref{equbigseq}) and intersect them with $L^{d-i}$, we get 
$$D^{j+1,i-j-1,0}(n)+\alpha_{j+1}^m\alpha_i^m\gamma^mD^{j,i-j-1,1}(n)> \epsilon^m \alpha_{j+1}^mD^{j,i-j,0}(n).$$
Dividing both side by $\epsilon^{(j+1)m}\beta_{j+1}^m$, we get 
$$\frac{D^{j+1,i-j-1,0}(n)}{\epsilon^{(j+1)m}\beta_{j+1}^m}+\frac{\alpha_i^m\gamma^mD^{j,i-j-1,1}(n)}{\epsilon^{(j+1)m}\beta_{j}^m}> \frac{D^{j,i-j,0}(n)}{\epsilon^{jm}\beta_{j}^m}.$$
Then we get 
$$\frac{D^{i,0,0}(n)}{\epsilon^{im}\beta_i^m}+\sum_{j=0}^{i-1}\frac{\alpha_i^m\gamma^mD^{j,i-j-1,1}(n)}{\epsilon^{(j+1)m}\beta_{j}^m}> D^{0,i,0}(n).$$
We note that $$D^{i,0,0}(n)=\deg_{i}(f^{(n+2)m}) \text{ and } D^{0,i,0}(n)=\deg_i(f^{(n+1)m}).$$
By Proposition \ref{proupperboundmixdeg}, we have 
$$D^{j,i-j-1,1}(n)\leq \binom{d}{j}\binom{d}{i-1}\frac{\deg_j(f^m)\deg_{i-1}(f^m)}{(L^d)^2}\deg_i(f^{mn}).$$
Then we get
$$\deg_i(f^{m(n+2)})+B(\alpha_1,\dots, \alpha_d; \gamma; \epsilon; m)\epsilon^{mi}\beta_i^m\deg_i(f^{mn})\geq \epsilon^{mi}\beta_i^m\deg_i(f^{m(n+1)}).$$

Now assume that conditions $J_i(\alpha_1,\dots, \alpha_d; \gamma; \epsilon; m), K_i(\alpha_1,\dots, \alpha_d; \gamma; \epsilon; m;N)$ are satisfied.
Write $$\phi:=\epsilon^{2mi}\beta_i^m \text{ and } \psi:=B(\alpha_1,\dots, \alpha_d; \gamma; \epsilon; m)\epsilon^{-mi}.$$
Condition $J_i(\alpha_1,\dots, \alpha_d; \gamma; \epsilon; m)$ implies that $$\epsilon^{mi}\beta_i^m>\phi+\psi.$$
Condition $K_i(\alpha_1,\dots, \alpha_d; \gamma; \epsilon; m; N)$ implies that
$$\deg_{i}(f^{m(N+1)})>\psi \deg_i(f^{mN}).$$
We then conclude the proof by Lemma \ref{lemrecuineq}.
\endproof

\medskip

\begin{lem}\label{lemrecuinequ}
Fix $i=1,\dots, d$. Assume that $\mu_i>\mu_{i+1}$.  Then by Remark \ref{remstricineqeta}, for every $j=0,\dots, i-1$, $\eta_{j,i-j-1}<\mu_{j+1}\mu_{i}.$
Pick $\gamma\in [\max_{j=0}^{i-1}\frac{\eta_{j,i-j-1}}{\mu_{j+1}\mu_{i}}, 1)$. Then for every $\epsilon\in (\gamma^{\frac{1}{3d}},1)$ there is $m_{\epsilon}>0$, such that for every $m\geq m_{\epsilon},$ conditions $I_i(\mu_1,\dots,\mu_d; \gamma; \epsilon; m)$, $J_i(\mu_1,\dots,\mu_d; \gamma; \epsilon; m)$ hold for $(X,f,L)$ and moreover, for every $N\geq 0$, $K_i(\mu_1,\dots, \mu_d; \gamma; \epsilon; m; N)$ is satisfied.
\end{lem}

\proof
The case $r_1=j, r_2=i-j-1$ of Theorem \ref{thmmaincompudegany} implies that the condition $I_i(\mu_1,\dots,\mu_d; \gamma; \epsilon; m)$ holds for $m\gg 0.$

As
\begin{flalign*}
\binom{d}{j}\binom{d}{j-1}\frac{\mu_i^m\gamma^m}{\epsilon^{m(j+1)}\la_j^m}\frac{\deg_j(f^m)\deg_{i-1}(f^m)}{(L^d)^2}
\approx& \frac{\mu_i^m\gamma^m}{\epsilon^{m(j+1)}\la_j^m}\la_j^m\la_{i-1}^m\\
\lesssim& \frac{\gamma^m}{\epsilon^{md}}\la_{i}^m,
\end{flalign*}
we get
\begin{equation}\label{equationbesti}
B(\mu_1,\dots, \mu_d; \gamma; \epsilon; m)\lesssim  \frac{\gamma^m}{\epsilon^{md}}\la_{i}^m.
\end{equation}

Then we have
\begin{flalign*}\frac{B(\mu_1,\dots, \mu_d; \gamma; \epsilon; m)}{\epsilon^{2mi}\la_i^m(1-\epsilon^{mi})}\lesssim(\frac{\gamma}{\epsilon^{3d}})^m.
\end{flalign*}
Since $\frac{\gamma}{\epsilon^{3d}}<1$, $J_i(\mu_1,\dots,\mu_d; \gamma; \epsilon; m)$ holds for $m\gg 0.$

By (\ref{equationbesti}),
\begin{flalign*}
\frac{\deg_{i}(f^{m})}{B(\alpha_1,\dots, \alpha_d; \gamma; \epsilon; m)\epsilon^{-mi}}\gtrsim \frac{\la_i^m\epsilon^{md}}{\frac{\gamma^m}{\epsilon^{md}}\la_{i}^m}=(\frac{\epsilon^{2d}}{\gamma})^m.
\end{flalign*}
So $K_i(\alpha_1,\dots, \alpha_d; \gamma; \epsilon; m; 0)$ is holds for $m\gg 0$.
Hence there is $m_{\epsilon}>0$, such that for every $m\geq m_{\epsilon},$ conditions $I_i(\mu_1,\dots,\mu_d; \gamma; \epsilon; m)$, $J_i(\mu_1,\dots,\mu_d; \gamma; \epsilon; m)$ and $K_i(\alpha_1,\dots, \alpha_d; \gamma; \epsilon; m; 0)$ hold for $(X,f,L)$.
By Lemma \ref{lemcondiimpliesdyna}, for every $m\geq m_{\epsilon}$ and $N\geq 0$, $K_i(\alpha_1,\dots, \alpha_d; \gamma; \epsilon; m; N)$ holds.
\endproof

\section{Algorithm to compute the dynamical degrees}\label{sectionalgdeg}
We keep the notations of Section \ref{secrecurdeg}.
Let $\bk$ be a field. Let $X$ be a projective variety of dimension $d$ over $\bk.$ Let $f: X\dashrightarrow X$ be a dominant rational self-map. Let $L$ be a big and nef line bundle in $\widetilde{\Pic}(X)$. 

\medskip

The aim of this section is to give an algorithm to compute the dynamical degrees to arbitrary precision.
In other words, we will give an algorithm, such that for any given number $l\in \Z_{>0},$ 
the algorithm gives a rational number $\tilde{\la}$ such that $\la_i\in (\tilde{\la}, \tilde{\la}+\frac{1}{2^l}).$
For a given precision $l$, the algorithm will stop in finitely many steps and it only uses certain intersection numbers between $L_n, n\geq 0$.

\subsection{Upper and lower bounds}
By \cite[Corollary 1.3]{Jiang2023} (or Theorem \ref{thmsiuin}), for every $m,n\geq 0$,  we have
\begin{equation}\label{equsubaddi}\deg_i (f^{m+n})\leq \frac{\binom{d}{i}}{(L^d)}\deg_i (f^{m})\deg_i (f^{n}).
\end{equation}
Then we get the following fact:
 \begin{fact}\label{factlaabove}
The sequence $$(\frac{\binom{d}{i}}{(L^d)}\deg_i (f^{n}))^{1/n}, n\geq 0$$ tends to $\la_i$ from above. 
 \end{fact}
 This controls the dynamical degrees from above. 
 
 \medskip
 
 Now we want to control the dynamical degrees from below. 
 By Remark \ref{remopencd}, for $i=1,\dots, d$, every $\beta\geq 1$, the following two statements are equivalent:
 \begin{points}
 \item there are decreasing numbers $\alpha_1,\dots, \alpha_d\in \R_{>0},$ with $\beta_i:=\prod_{j=1}^i\alpha_j>\beta,$ $\gamma\in (0,1)$, $\epsilon\in((\frac{\beta}{\beta_i})^{\frac{1}{2i}},1)$ and $m\in \Z_{\geq 1}$, such that the conditions $I_i(\alpha_1,\dots, \alpha_d; \gamma, \epsilon, m)$, $J_i(\alpha_1,\dots, \alpha_d; \gamma, \epsilon, m)$ and $K_i(\alpha_1,\dots, \alpha_d; \gamma, \epsilon, m; 0)$ are satisfied for $(X,f,L)$;
\item there are decreasing numbers $\alpha_1,\dots, \alpha_d\in \Q_{>0},$ with $\beta_i:=\prod_{j=1}^i\alpha_j>\beta$, $\gamma\in (0,1)\cap \Q$,  $\epsilon\in((\frac{\beta}{\beta_i})^{\frac{1}{2i}},1)\cap \Q$ and $m\in \Z_{\geq 1}$; such that the conditions $I_i(\alpha_1,\dots, \alpha_d; \gamma, \epsilon, m), J_i(\alpha_1,\dots, \alpha_d; \gamma, \epsilon, m)$ and $K_i(\alpha_1,\dots, \alpha_d; \gamma, \epsilon, m; 0)$ are satisfied for $(X,f,L)$;
 \end{points}

 By Lemma \ref{lemrecuinequ}, Lemma \ref{lemcondiimpliesdyna}, we have the following result.
\begin{thm}\label{thmkeyestlowb}
For $i=1,\dots, d$, every $\beta\geq 1$, (i) (or (ii)) above implies that $\la_i>\beta.$
On the other hand, if $\mu_i>\mu_{i+1}$, and $\la_i>\beta$ then (i) (and (ii)) holds.
\end{thm}

Combining Fact \ref{factlaabove} with Theorem \ref{thmkeyestlowb} and the log concavity of the dynamical degrees, we can control $\la_i$ both from above and form below. This give us an algorithm to compute $\la_i$ which only use certain intersection numbers. Now we  explain the algorithm in more details.

\medskip

\subsection{Algorithm}\label{subsectalgo}

\proof[Proof of Theorem \ref{thmcomputela}]
%Set $D:=\max_{j=0}^d\{\frac{\binom{d}{j}}{(L^d)}(\deg_j (f))\}.$
For every $t\geq 0$, and $i=0,\dots, d$, we will construct upper bounds $\la_i^+(t)$ and lower bounds $\la_{i}^-(t)$ for $\la_i$. 
First, we set $\la_0^+(t)=\la_0^-(t)=1$ for every $t\geq 0.$ 

\medskip

We first construct the upper bounds for $i=1,\dots, d$. For $t\geq 1$, define $$U(t):=(\frac{\binom{d}{i}}{(L^d)}\deg_i (f^{t}))^{1/t}$$ and set $$\la_i^+(t):=\min\{U(0),\dots, U(t)\}.$$
By Fact \ref{factlaabove}, $\la_i^+(t)$ tends to $\la_i$ from above. 

\medskip

Next, we construct the lower bounds for $\la_{i}^-(t), i=1,\dots, d, t\geq 0$ inductively.
Set $\la_{i}^-(0):=1, i=1,\dots, d.$
Let $\Omega$ be the set of $(\alpha_1,\dots, \alpha_d,\gamma, \epsilon, m)\in \Q_{>0}^{d}\times((0,1)\cap \Q)^2\times \Z_{>0}$.
This is a countable set.  
We may fix a (computable) arrangement to write $\Omega=\{\omega_t, t\geq 1\}.$
Write $\omega_t=(\alpha_1(t),\dots, \alpha_d(t),\gamma(t), \epsilon(t), m(t))$.
For $i=0,\dots, d$, set  $\beta_i(t):=\prod_{j=1}^i\alpha_j(t)$.  Not that $\beta_0(t)=1.$
Define $$\beta_i^-(t)=\epsilon(t)^{2i}\beta_i(t)$$ if the conditions $I_i, J_i$ hold for $(\alpha_1(t),\dots, \alpha_d(t); \gamma(t), \epsilon(t), m(t))$ and the condition $K_i$ holds for $(\alpha_1(t),\dots, \alpha_d(t); \gamma(t), \epsilon(t), m(t); 0)$; and 
$$\beta^{-}_i(t):=1$$ otherwise.
To compute $\beta^{-}_i(t)$, we only need to compute finitely many intersection numbers among $L,L_{m(t)}, L_{2m(t)}$.
By Theorem \ref{thmkeyestlowb}, we have $\la_i\geq \beta^{-}_i(t)$. Define $$X_i(t):=\max\{\la_i^{-}(t-1),\beta_i^-(t)\}.$$
Then we have $\la_i\geq X_i(t).$
Define $\la^-_d(t):=X_d(t).$
For $i=1,\dots, d-1$, define
$$\la^-_i(t):=\max\{X_i(t), \max_{1\leq u\leq i, 1\leq v\leq d-i}\{(X_{i-u}(t)^vX_{i+v}^u)^{1/(u+v)}\}\}.$$
As $\la_i, i=0,\dots, d$ is log concave, we have $\la_i\geq \la^-_i(t).$ 
Note that $\la^-_i(t)$ is increasing when $t$ increase. 
For each $t\geq 0$, to compute $\la^-_i(t)$, we only need to use finitely
 many mixed degrees.

\medskip

Now we compute $\la_i^-(t),\la^+_i(t), i=0,\dots, d$ for $t\geq 0$ one by one until $t$ reach some value $T$ such that
$$\la_i^+(T)-\la^-_i(T)<1/2^l$$ for every $i=0,\dots, d.$ Then we set $\widetilde{\la_i}:=\la^-_i(T).$ Once such $T$ exists, we get $$\tilde{\la}<\la_i\leq\la_i^+(T)<\tilde{\la}+\frac{1}{2^l}.$$ So the output $\tilde{\la}$ is what we need.

\medskip

Now we only need to prove that there is $t\geq 0$ such that $\la_i^+(t)-\la^-_i(t)<1/2^l$ for every $i=0,\dots, d.$
As $$\lim_{t\to \infty}\la_i^+(t)=\la_i$$ for every $i\geq 0$, we only need to show that
 \begin{equation}\label{equlowerlaiag}\lim_{t\to \infty}\la_i^-(t)=\la_i.
 \end{equation}
Let $V:=\{0\}\cup \{i=1,\dots, d|\,\, \mu_{i+1}<\mu_i\}.$
It is clear that $\{0,d\}\subseteq V.$
By Theorem \ref{thmkeyestlowb},  for every $i\in V$, (\ref{equlowerlaiag}) holds. 
For every $i\in \{0,\dots, d\}\setminus S$, set $$a_i:=\max\{j\in V| j<i\}$$ and  $$b_i:=\min\{j\in V| j>i\}.$$ 
Then $\mu_j$ is constant when $j\in [a_i+1,b_i].$
So we have $$\la_i=(\la_{a_i}^{b_i-i}\la_{b_i}^{i-a_i})^{1/(b_i-a_i)}.$$
The definition of $\la_i^{-}(t)$ implies that $$\la_i^{-}(t)\geq (\la^-_{a_i}(t-1)^{b_i-i}\la^-_{b_i}(t-1)^{i-a_i})^{1/(b_i-a_i)}.$$
As $a_i, b_i\in V$, we get 
$$\liminf_{t\to\infty} \la_i^{-}(t)\geq \lim_{t\to \infty} (\la^-_{a_i}(t-1)^{b_i-i}\la^-_{b_i}(t-1)^{i-a_i})^{1/(b_i-a_i)}=\la_i.$$
As $\la_i^{-}(t)\leq \la_i$ for every $t\geq 0$, (\ref{equlowerlaiag}) holds. This concludes the proof.
\endproof
\subsection{Lower bounds in dimension two}\label{subsecdimtwo}
By Fact \ref{factlaabove}, we have a direct upper bound of dynamical degrees, but to get lower bounds we need to try many possible parameters to see whether $\theta$ equal to $1.$
This makes the algorithm in Section \ref{subsectalgo} far from being efficient.  I suspect that a direct way to get the lower bounds should make the algorithm more efficient.  In the surface case, a such lower bound was proved by the author \cite[Key Lemma]{Xie2015}.
\begin{thm}[=Theorem \ref{thmxiedukeklin}]\label{thmxiedukekl}
Let $\bk$ be a field. Let $X$ be a projective surface over $\bk.$ Let $f: X\dashrightarrow X$ be a dominant rational self-map. Let $L$ be a big and nef line bundle in $\widetilde{\Pic}(X).$ Then we have $$\la_1\geq \frac{\deg_1 f^2}{2^{\frac{1}{2}}\times 3^{18}\deg_1 f}.$$
\end{thm}
The proof relies on the theory of hyperbolic geometry and the natural linear action of $f$ on a suitable hyperbolic space of infinite dimension. This space is constructed as a set of cohomology classes in the Riemann-Zariski space of $X$ and was introduced by Cantat \cite{Cantat2007}. Unfortunately, such space can be only constructed in dimension two. Also the coefficient $2^{\frac{1}{2}}\times 3^{18}$ is quite large. 

\medskip

In this section, we use the idea of constructing recursive inequalities to get a better lower bound for the first dynamical degree in dimension two. This result has the same form as Theorem \ref{thmxiedukekl}, but it improves the coefficient a lot i.e. from $2^{\frac{1}{2}}\times 3^{18}$ to $4$. Moreover, the proof become much simpler.
\begin{thm}[=Theorem \ref{thmsurfaceboundin}]\label{thmsurfacebound}Let $\bk$ be a field. Let $X$ be a projective surface over $\bk.$ Let $f: X\dashrightarrow X$ be a dominant rational self-map. Let $L$ be a big and nef line bundle in $\widetilde{\Pic}(X).$ Then we have $$\la_1\geq \frac{\deg_1 f^2}{4\deg_1 f}.$$
\end{thm}

\proof
Set $Q:=\frac{\deg_1 f^2}{\deg_1 f}.$
 As $\la_1\geq \la_2^{1/2}$, we may assume that $Q/4\geq \la_2^{1/2}$.
We claim that the line bundle $$L_2+\frac{Q^2}{16}L-\frac{Q}{2}L_1$$ is big.
For this, by Theorem \ref{thmsiuincodimo}, we only need to show that 
\begin{equation}\label{equsurfabig}\frac{(L_2+\frac{Q^2}{16}L)^2}{2(L_2+\frac{Q^2}{16}L)L_1}\geq \frac{Q}{2}.
\end{equation}
Indeed,
\begin{flalign*}\frac{(L_2+\frac{Q^2}{16}L)^2}{2(L_2+\frac{Q^2}{16}L)L_1}\geq& \frac{2\frac{Q^2}{16}\deg_1 f^2}{2(L_2+\frac{Q^2}{16}L)L_1}\\
\geq& \frac{2\frac{Q^2}{16}\deg_1 f^2}{2(\la_2\deg_1f+\frac{Q^2}{16}\deg_1 f)}\\
\geq & \frac{2\frac{Q^2}{16}\deg_1 f^2}{4\frac{Q^2}{16}\deg_1 f}\\
\geq & \frac{Q}{2}
\end{flalign*}
Then we get (\ref{equsurfabig}).
Apply $(f^n)^*$ to $L_2+\frac{Q^2}{16}L-\frac{Q}{2}L_1$ and multiply it by $L$, we get that for every $n\geq 0$,
$$\deg_1(f^{n+2})+\frac{Q^2}{16}\deg_1(f^n)\geq \frac{Q}{2}\deg_1 (f^{n+1}).$$
By (\ref{equsubaddi}), $\deg_1f^2\leq \frac{2}{(L^2)}(\deg_1 f)^2$. Then
we have $$\deg_1f-\frac{Q}{4}(L^2)\geq \frac{Q}{2}(L^2)-\frac{Q}{4}(L^2)>0.$$
We concludes the proof by Lemma \ref{lemrecuineq}.
\endproof

\section{Lower semi-continuity of dynamical degrees}\label{sectionsslsc}
Let $S$ be an integral noetherian scheme. Recall that 
A \emph{family of $d$-dimensional dominant rational self-maps on $S$} is a flat and projective scheme $\pi:\sX\to S$ satisfying $d:=\dim \sX/S$ with  a dominant rational self-map
$f: \sX\dashrightarrow \sX$ over $S$ such that the following hold:
\begin{points}
\item  For every $p\in S$, the fiber $X_p$ of $\pi$ at $p$ is geometrically reduced and irreducible.
\item  For every point $p\in S$, $X_p\not\subseteq I(f)$. 
\item The induced map $f_p: X_p\dashrightarrow X_p$ is dominant. 
\end{points}

\medskip

The aim of this section is to prove the lower semi-continuity of dynamical degrees for a family of dominant rational maps.
\begin{thm}[=Theorem \ref{thmlscdynadegin}]\label{thmlscdynadeg}
Let $S$ be an integral noetherian scheme and $\pi:\sX\to S$ be a flat and projective scheme over $S$. 
Let $f: \sX\dashrightarrow \sX$ be a family of $d$-dimensional dominant rational self-maps on $S$. Then
for every $i=0,\dots,d$, the function $p\in S\mapsto \lambda_i(f_p)$ is lower semi-continuous.
\end{thm}

%A special case of Theorem \ref{thmlscdynadeg} is the following result.
%\begin{cor}[=Corollary \ref{corlscdynadegoverzin}]\label{corlscdynadegoverz}
%Let $f: \P^d_{\Z}\dashrightarrow \P^d_{\Z}$ be a dominant rational map over $\Z$. 
%Then for every $i=0,\dots,d$, we have 
%$$\la_i(f)=\lim_{p \text{ prime}, p\to \infty} \la_i(f_p).$$
%\end{cor}
%
%
%
%Theorem \ref{thmlscdynadeg} generalizes \cite[Theorem 4.3]{Xie2015} from dimension two to any dimension.
%The special case where $i=1$ and $\sX=\P^N_S$ of Theorem \ref{thmlscdynadeg} implies Call-Silverman's conjecture \cite[Conjecture 1]{Silverman2018} and its generalized version \cite[Conjecture 14.13]{Benedetto2019}. Corollary \ref{corlscdynadegoverz} gives a positive answer to \cite[Question 14.10]{Benedetto2019}.
%
%
%In \cite[Section 4.3]{Xie2015}, the author  provided the following example to show that Corollary \ref{corlscdynadegoverz} cannot be strengthened to the statement that 
%$\la_i(f_{\kappa})=\la_i(f_p)$ for infinitely many prime $p$.
%\exe
%Let $f: \P^2_{\Z}\dashrightarrow \P^2_{\Z}$ sending $[x:y:z]$ to $[xy: xy-2z^2: yz+3z^2].$
% Then $\la_1(f\times_{\Z}\Q)=2$, but $\la_1(f_p)<2$ for all prime $p.$
%\endexe

\subsection{Lower semi-continuity functions on noetherian schemes}
%Recall that $|S|$ is the underling set of $S$ with the constructible topology.
The following lemma gives
a criterion for the lower semi-continuity.
\begin{lem}\label{lemlsconneoth}Let $S$ be a noetherian scheme. Then a function $\theta: S\to \R$ is lower semi-continuous if and only if the followings hold:
\begin{points}
\item for points $x,y\in S$ with $y\in \overline{\{x\}}$, we have $\theta(x)\geq \theta(y);$
\item for every $x\in S$ and $a<\theta(x)$, there is an open subset $V$ of $\overline{\{x\}}$ containing $x$ such that $V\subseteq \theta^{-1}((a,+\infty)).$
\end{points}
\end{lem}
\proof First assume that $\theta$ is lower semi-continuous. Then (ii) is obvious.
Let $x,y\in S$ with $y\in \overline{\{x\}}$. Note that $\theta^{-1}((-\infty, \theta(x)])$ is closed and it contains $x$. Then we have $y\in \overline{\{x\}}\subseteq \theta^{-1}((-\infty, \theta(x)]).$
So (i) holds. 

Now assume that (i) and (ii) hold. Let $a\in \R.$ Let $Z:=\overline{\theta^{-1}((-\infty,a])}.$ We only need to show that $Z=\theta^{-1}((-\infty,a]).$
Otherwise, there is an irreducible component $Z'$ of $Z$ such that $Z'\not\subseteq \theta^{-1}((-\infty,a]).$ Let $\eta$ be the generic point of $Z'.$
By (i), $\eta\not\in \theta^{-1}((-\infty,a]).$ By (ii), there is open subset $V$ of $Z'$ containing $\eta$ such that $V\subseteq \theta^{-1}((a,+\infty)).$
So $V\not\subseteq \theta^{-1}((-\infty,a]).$ So $\theta^{-1}((-\infty,a])$ is not dense in $Z'$, which is a contradiction.
\endproof

\begin{rem}\label{remlimlsc}
The following example shows that the limit of lower semi-continuous functions may not be lower semi-continuous:
Let $S=\Spec \Z$. Let $\eta$ be the generic point of $S$. For $n\geq 1$, let $D_n: S\to \R$ be the function as follows: Define $D_n(\eta):=1.$ for every prime number $p$,
$D_n(p):=0$ if $p<n$; and $D_n(p):=1$ if $p\geq n.$  
Easy to check that $D_n$ are lower semi-continuous.
Easy to see that $D_n$ pointwisely converges to the function $D:S\to \R$ satisfying $D(\eta)=1$ and $D|_{S\setminus \{\eta\}}=0$, which is not lower semi-continuous.
\end{rem}

\subsubsection{Constructible topology}	
Let $S$ be a noetherian scheme.
Denote by $|S|$ the underling set of $S$ with the constructible topology; i.e. the topology on a  $S$ generated by the constructible subsets (see~\cite[Section~(1.9) and in particular (1.9.13)]{EGA-IV-I}).
In particular every constructible subset is open and closed.
This topology is finer than the Zariski topology on $S.$ Moreover $|S|$ is (Hausdorff) compact.

\begin{lem}\label{lemlscimconcon}Let $S$ be a noetherian scheme. Let $\theta: S\to \R$ be a lower semi-continuous function. Then $\theta$ is continuous in the constructible topology. 
Assume further that $\theta(S)$ is discrete. Then $\theta(S)$ is finite and for every $a\in \R$,  $\theta^{-1}(a)$ is a constructible subset of $S.$
\end{lem}
\proof
For every $a\in \R$, $\theta^{-1}((a,+\infty))$ is Zariski open in $S$, hence open in the constructible topology.
For every $a\in \R$, $\theta^{-1}((-\infty,a))=\cup_{n\geq 1}\theta^{-1}((-\infty,a-\frac{1}{n}]).$
As $\theta^{-1}((-\infty,a-\frac{1}{n}])$ is Zariski closed, it is open in the constructible topology.
So $\theta^{-1}((-\infty,a))$ is open in the constructible topology. So $\theta$ is continuous in the constructible topology.

Now assume that $\theta(S)$ discrete. As $\theta$ is continuous on $|S|$ and $|S|$ is compact, $\theta(S)$ is compact and discrete, hence finite.
If $a\not\in \theta(S)$, then $\theta^{-1}(a)=\emptyset.$
Assume that $a\in \theta(S)$. There is $b<a$ such that $(b,a)\cap \theta(S)=\emptyset.$
Then $$\theta^{-1}(a)=\theta^{-1}((-\infty,a])\setminus \theta^{-1}((-\infty,b]),$$ which is a constructible set.
\endproof
\subsection{Lower semi-continuity of mixed degrees}
Let $\sL$ be a $\pi$-ample line bundle on $\sX$. For every $p\in S$, denote by $L_p$ the restriction of $\sL$ to the fiber $X_p.$

\begin{lem}\label{lowersemicontinuous}
Let $(\sX, f)$ be a family of $d$-dimensional dominant rational self-maps on $S$.
Let $\sL$ be a $\pi$-ample line bundle on $\sX$. 
Let $m_1,\dots, m_d\in \Z_{\geq 0}.$
Then the function $$p\mapsto ((f_p^{m_1})^*L_p\cdots (f_p^{m_d})^*L_p)$$ is lower semi-continuous on $S$.
In particular, for every $i=0,\dots,d$, the function $$p\mapsto \deg_{i,L_p}f_p$$ is lower semi-continuous on $S$.
\end{lem}

\proof[Proof of Lemma \ref{lowersemicontinuous}]Denote by $\kappa$ the generic point of $S$.
By Lemma \ref{lemlsconneoth}, we only need to show that for any $(\pi: \sX\to S, f, \sL)$ satisfying our assumption, we have
$$((f_p^{m_1})^*L_p\cdots (f_p^{m_d})^*L_p)\leq ((f_{\kappa}^{m_1})^*L_{\kappa}\cdots (f_{\kappa}^{m_d})^*L_{\kappa})$$ on $S$ with equality on a
Zariski open subset of $S$.

\medskip

Let $\Gamma_{\kappa}$ be the closure of the image of the map $X_{\kappa}\dashrightarrow X_{\kappa}^d$ sending $x$ to $(f_{\kappa}^{m_1}(x),\dots,f_{\kappa}^{m_d}(x))$.
Let $\Gamma$ be its closure in $\sX_{/S}^d.$

By \cite[Theorem 5.2.2]{Raynaud1971}, there is a blowup $\phi:S'\to S$ such that the strict transformation $\Gamma'\to S'$ of $\Gamma\to S$ by $\phi$ is flat. Set $\sX':=\sX\times_S S'$ with structure morphism $\pi': \sX'\to S'$ and $f':=f\times_S \id$.  
Set $\psi:=\id\times_S \phi: \sX'\to \sX$ and $\sL':=\psi^*\sL.$
Then $(\pi': \sX'\to S', f', \sL')$ has the same property as $(\pi: \sX\to S, f,\sL)$.
Let $\kappa'$ be the generic point of $S'$ and $\Gamma_{\kappa}'$ be the closure of the image of the map ${X'}_{\kappa'}\dashrightarrow {X'}_{\kappa'}^d$ sending $x$ to $((f_{\kappa'}')^{m_1}(x),\dots,(f_{\kappa'}')^{m_d}(x))$.
Then $\Gamma'$ is its closure in ${\sX'}_{/S}^d.$ 
For every $p'\in S'$, $(X'_{p'},f_{p'}', L_{p'}')$ is a bass change of $(X_{p},f_{p}, L_{p}).$
Moreover $\psi$ is an isomorphism over a Zariski dense open subset of $S.$
So we may replace $(\pi: \sX\to S, f,\sL)$ by $(\pi': \sX'\to S', f', \sL')$ and assume further that the structure morphism $\pi_{\Gamma}: \Gamma\to S$ is flat.

For every $p\in S$, let $\Gamma_p''$ be the closure of the image of the map $X_{p}\dashrightarrow  X_p^d$ sending $x$ to $(f^{m_1}(x),\dots,f^{m_d}(x))$.
Then $\Gamma_p''$ is an irreducible component of $\Gamma_p.$ There is a Zariski dense open subset $U$ of $S$ such that for every $p\in U$, $\Gamma_p=\Gamma_p''.$
Let $F_i: \Gamma\to \sX$ be the $i$-th projection. Then we have 
\begin{equation}\label{equintermixd}
((f_p^{m_1})^*L_p\cdots (f_p^{m_d})^*L_p)=(F_1^*\sL|_{\Gamma_{p}''}\cdots F_d^*\sL|_{\Gamma_{p}''})\leq (F_1^*\sL|_{\Gamma_{p}}\cdots F_d^*\sL|_{\Gamma_{p}}),
\end{equation}
and the equality holds for $p\in U.$
By \cite[Proposition 10.2]{Fulton1984}, we have
\begin{equation}\label{equinterconst}
(F_1^*\sL|_{\Gamma_{p}}\cdots F_d^*\sL|_{\Gamma_{p}})=(F_1^*\sL|_{\Gamma_{\kappa}}\cdots F_d^*\sL|_{\Gamma_{\kappa}})=((f_{\kappa}^{m_1})^*L_{\kappa}\cdots (f_{\kappa}^{m_d})^*L_{\kappa}).
\end{equation}
Combine (\ref{equinterconst}) with (\ref{equintermixd}), we concludes the proof.
\endproof

\rem\label{remmixdegcon}For every $p\in S$, we have $((f_p^{m_1})^*L_p\cdots (f_p^{m_d})^*L_p)\in \Z.$ By Lemma \ref{lemlscimconcon}, 
the set $\{((f_p^{m_1})^*L_p\cdots (f_p^{m_d})^*L_p)|\,\, p\in S\}$ is finite and for every subset $F\subseteq \R$, $\{p\in S|\,\, ((f_p^{m_1})^*L_p\cdots (f_p^{m_d})^*L_p)\in F\}$
is a constructible subset of $S.$
\endrem

\medskip
The following example shows that the map $p\mapsto \deg_{L_p,i}(f_p)$ is not continuous in general.
\begin{exe}\cite[Example 4.2]{Xie2015} Consider the birational transformation $$f: [x:y:z]\mapsto [xz:yz+2xy:z^2]$$ of $\mathbb{P}^2$ over $\Spec \mathbb{Z}$. Denote by $L$ the hyperplane line bundle on $\mathbb{P}^2_{\mathbb{Z}}$. Then $f_p$ is birational for every prime $p\in \Spec \Z$. We have that
$\deg_{L_p}(f_p) = 1$ for $p=2$ and $\deg_{L_p}(f_p) = 2$ for any odd prime.
\end{exe}

\medskip

The function $p\in S\mapsto \lambda_i(f_p)$ of the $i$-th dynamical degree is the point-wise limit of the functions $p\in S\mapsto (\deg_{i,L_p}f_p)^{1/n}$.
By Lemma \ref{lowersemicontinuous}, the later function is lower semi-continuous.
However, as shown in Remark \ref{remlimlsc}, this does not directly imply the lower semi-continuity of the $i$-th dynamical degree.
To complete the proof of Theorem \ref{thmlscdynadeg}, we need to apply the lower bounds of the dynamical degrees obtained in Section \ref{secrecurdeg}.

\subsection{Lower semi-continuity of dynamical degrees}
\proof[Proof of Theorem \ref{thmlscdynadeg}]
Let $\kappa$ be the generic point of $S.$
By Lemma \ref{lemlsconneoth}, we only need to show that for any $(\pi: \sX\to S, f)$ satisfying our assumption, the followings hold:
\begin{points}
\item $\lambda_i(f_p) \leq \lambda_i(f_{\kappa})$ for all $p\in S$ 
\item for any $\beta < \lambda_i(f_{\kappa})$,
there is a nonempty open set $U$ of $S$,
such that for every point $p\in U$, $\lambda_i(f_{p})> \beta$.
\end{points}

Let $\sL$ be a $\pi$-ample line bundle on $\sX$. For every $p\in S$, denote by $L_p$ the restriction of $\sL$ to the fiber $X_p.$
By Lemma \ref{lowersemicontinuous}, for every integer $n>0$, we have $$\deg_{i,L_p}(f^n_p)\leq\deg_{i,L_{\kappa}}(f^n_{\kappa})$$ hence
$$\lambda_i(f_p)\leq \lambda_i(f_{\kappa}).$$ This implies (i).

Set $V:=\{0\}\cup \{i=1,\dots, d|\,\, \mu_{i+1}(f_{\kappa})<\mu_i(f_{\kappa})\}.$ It is clear that $\{0,d\}\subseteq V.$
We first prove (ii) for $i\in V.$
Let $\beta<\la_i(f_{\kappa}).$ By Theorem \ref{thmkeyestlowb}, 
there are decreasing numbers $\alpha_1,\dots, \alpha_d\in \R_{>0},$ with $\beta_i:=\prod_{j=1}^i\alpha_j>\beta,$ $\gamma\in (0,1)$, $\epsilon\in((\frac{\beta}{\beta_i})^{\frac{1}{2i}},1)$ and $m\in \Z_{\geq 1}$, such that the conditions $I_i(\alpha_1,\dots, \alpha_d; \gamma, \epsilon, m)$, $J_i(\alpha_1,\dots, \alpha_d; \gamma, \epsilon, m)$ and $K_i(\alpha_1,\dots, \alpha_d; \gamma, \epsilon, m; 0)$ are satisfied for $(X_{\kappa},f_{\kappa},L_{\kappa})$.
Note that the conditions $I_i(\alpha_1,\dots, \alpha_d; \gamma; \epsilon; m)$, $J_i(\alpha_1,\dots, \alpha_d; \gamma; \epsilon; m)$ and $K_i(\alpha_1,\dots, \alpha_d; \gamma; \epsilon; m;0)$ for $(X_p,f_p, L_p)$ only depend on the top intersection numbers of $(f_p^{2m})^*L_p, (f_p^{m})^*L_p, L_p$. 
By Lemma \ref{lowersemicontinuous}, there is a Zariski dense open subset $U$ of $S$, such that for every $p\in U$, all top intersection numbers of $(f_p^{2m})^*L_p, (f_p^{m})^*L_p, L$ are constant. Hence for every $p\in U$, the conditions $I_i(\alpha_1,\dots, \alpha_d; \gamma, \epsilon, m)$, $J_i(\alpha_1,\dots, \alpha_d; \gamma, \epsilon, m)$ and $K_i(\alpha_1,\dots, \alpha_d; \gamma, \epsilon, m; 0)$ are satisfied for $(X_{p},f_{p},L_{p})$. By Theorem \ref{thmkeyestlowb}, we get $\lambda_i(f_p)> \beta.$ This implies (ii) for $i\in V.$

\medskip

As in the proof of Theorem \ref{thmcomputela},
for every $i\in \{0,\dots, d\}\setminus V$, set $$a_i:=\max\{j\in V| j<i\}$$ and  $$b_i:=\min\{j\in V| j>i\}.$$ 
We have $$\la_i(f_{\kappa})=(\la_{a_i}(f_{\kappa})^{b_i-i}\la_{b_i}(f_{\kappa})^{i-a_i})^{1/(b_i-a_i)}.$$
Pick $\theta\in (\beta/\la_i(f_{\kappa}), 1)$.
As $a_i, b_i\in V$, there is a nonempty open set $U$ of $S$,
such that for every point $p\in U$, $$\lambda_{a_i}(f_{p})> \theta \lambda_{a_i}(f_{\kappa})$$ and 
$$\lambda_{b_i}(f_{p})> \theta \lambda_{b_i}(f_{\kappa}).$$
As $\la_i(f_p), i=0,\dots, d$ is log concave, 
we have $$\la_i(f_p)\geq (\la_{a_i}(f_{p})^{b_i-i}\la_{b_i}(f_{p})^{i-a_i})^{1/(b_i-a_i)}$$$$\geq \theta(\la_{a_i}(f_{\kappa})^{b_i-i}\la_{b_i}(f_{\kappa})^{i-a_i})^{1/(b_i-a_i)}$$$$=\theta \la_i(f_{\kappa})>\beta.$$
This concludes the proof.
\endproof

\medskip

Theorem \ref{thmlscdynadeg}  implies that for every family of dominant rational self-maps over $S$, $\la_i(f_p)$ can not be arbitrarily closed to $1$ if it is not equal to $1.$
\begin{cor}\label{corwellorded}Let $(\sX, f)$ be a family of $d$-dimensional dominant rational maps on $S$. For every $i=0,\dots,d,$ the set $\Lambda_i((\sX, f)):=\{\la_i(f_p)|\,\, p\in S\}$ is well-ordered i.e. every subset of $\Lambda_i((\sX, f))$ has a minimal element. In particular, there is $\la\in (1,+\infty)$ such that 
for every $p\in S$, if $\la_i(f_p)>1$, then $\la_i(f_p)\geq \la.$
\end{cor}
\proof
Let $F\subseteq \Lambda_i((\sX, f))$. For every $\beta\in \R$, define $Z_{\beta}:=\la_i^{-1}((-\infty, \beta])$ which is Zariski closed. 
For $\beta\in \Lambda_i((\sX, f))$, we have 
\begin{equation}\label{equabetasup}\beta=\sup_{p\in Z_{\beta}}\la_i(f_p).
\end{equation}
Set $b:=\inf F$. There is a decreasing sequence $\beta_n\in F$ such that $\lim\limits_{n\to \infty}\beta_n=b.$
The noetherianity of $S$ shows that there is $N\geq 0$ such that $Z_{\beta_n}=Z_{\beta_N}$ for all $n\geq N.$
By (\ref{equabetasup}), we get $\beta_n=\beta_N$ for every $n\geq N$. Hence $b=\beta_N$. This implies that $\Lambda_i((\sX, f))$ is well-ordered. 
As $\Lambda_i((\sX, f))\cap (1,+\infty)$ is well-ordered,
the last statement is true.
\endproof

\subsection{Decidability}
Theorem \ref{thmlscdynadeg} implies that for a family of dominant rational self-maps over $S$. For every $\beta>0$, $i=1,\dots, d$, and $p\in S$, the question whether 
$\la_i(f_p)>\beta$ is decidable. 

\medskip
Let $(\sX, f)$ be a family of $d$-dimensional dominant rational maps on $S$. Let $\sL$ be a $\pi$-ample line bundle on $\sX$. 
For every $p\in S$, denote by $X_p, f_p, L_p$ the specialization of $\sX, f,L$ at $p.$

For $t\geq 0$, let $\la_i^-(f_p, t)$ be the lower bounds for $\la_i(f_p)$ as defined in the proof of Theorem \ref{thmcomputela} in Section \ref{subsectalgo}.
It has the following properties:
\begin{points}
\item for every $i=0,\dots, d$, $\la_i^-(f_p, t), t\geq 0$ is increasing;
\item $\lim_{t\to \infty}\la_i^-(f_p, t)=\la_i(f_p)$;
\item for every $i,t$, the value of $\la_i^-(f_p, t)$ only relies on finitely many mixed degrees for $X_p, f_p, L_p$.
\end{points}

For every $t\geq 0$, by (iii) and Remark \ref{remmixdegcon}, the function 
$p\in S\mapsto \la_i^-(f_p, t)$ is locally constant on $|S|$. In particular, it only takes finitely many values.

%For $t\geq 0$, define $\Theta_t: S\to \{0,1\}$ as follows:
%$\Theta(p,t):=1$ if $\la_i^-(f_p, t)>\beta$.

%To compute $\Theta(p,\omega)$, we only need to compute finitely many intersection numbers between $(f_p^{2m})^*L_p, (f_p^{m})^*L_p, L_p$.

\begin{cor} Let $(\sX, f)$ be a family of $d$-dimensional dominant rational maps on $S$. Let $\sL$ be a $\pi$-ample line bundle on $\sX$. 
Then for every $\beta\in \R$ and $i=0,\dots, d$, there is $T\geq 0$ such that for every $p\in S$, $\la_i(f_p)>\beta$ if and only if $\la_i^-(f_p, T)>\beta.$
\end{cor}
\proof
Set $Z:=\{p\in S|\,\, \la_i(f_p)> \beta\}.$ 
By Theorem \ref{thmlscdynadeg}, $Z$ is Zariski open. 
For every $n\geq 1$, write $V_t:=\{p\in S|\,\, \la_i^-(f_p, t)>\beta\}.$ By (i) and (ii) above, $V_t$ is increasing and $Z=\cup_{t\geq 0}V_t.$
By (iii) above and Remark \ref{remmixdegcon}, $V_t$ is constructible in $S$. As $|Z|$ is compact and $|V_t|, t\geq 0$ are open in the constructible topology, there is $T\geq 0$ such that 
$Z=\cup_{t=0}^TV_t=V_T.$ This concludes the proof.
\endproof

\section{Periodic points of cohomologically hyperbolic maps}\label{sectioncohhyp}
Let $X$ be a variety over a field $\bk$. Let $f: X\dashrightarrow X$ be a dominant rational self-map. 
For $i=1,\dots,d$, we say that $f$ is \emph{$i$-cohomologically hyperbolic} if $\la_i(f)$ is strictly larger than the other dynamical degrees i.e. $$\mu_i(f)>1 \text{ and } \mu_{i+1}(f)<1.$$
We say that $f$ is \emph{cohomologically hyperbolic} if it is $i$-cohomologically hyperbolic for some $i=1,\dots,d$ i.e. $\mu_j(f)\neq 1$ for every $j=1,\dots,d.$

\medskip

Let $X_{f}$ be the set of (scheme-theoretic) points $x$ whose orbit is well-defined i.e. for every $n\geq 0$, $f^n(x)\not\in I(f).$
More generally, for every Zariski open subset $V$ of $X$, let $V_f$ be the set of points $x\in X_f$ whose orbit $O_f(x)$ is contained in $V.$
Let $X_{f}(\overline{\bk}):=X_f\cap X(\overline{\bk})$.
For every $n\geq 0$, let $\Per_n(f)(\overline{\bk})$ be the set of $n$-periodic closed points in $X_{f}(\overline{\bk})$.
Set $\Per(f)(\overline{\bk}):=\cup_{n\geq 1}\Per_n(f)(\overline{\bk}).$
For every Zariski open subset $V$ of $X$, let $\Per_V(f)(\overline{\bk})$ be the set of $x\in \Per(f)(\overline{\bk})$ whose orbit $O_f(x)$ is contained in $V.$

\medskip

The aim of this section is to prove the following result, which implies Theorem \ref{thmperiodisodenseint}.
\begin{thm}\label{thmperiodisodense} If $f$ is cohomologically hyperbolic, then for every Zariski dense open subset $V$ of $X$, $\Per_V(f)(\overline{\bk})$ is Zariski dense in $X$.
\end{thm}

\subsection{Rational self-maps over finite fields}
The following result shows that, for dominant rational self-maps over finite fields, the periodic points are always Zariski dense. 
It was originally proved by Fakhruddin and Poonen \cite[Proposition 5.5]{fa} for endomorphisms.  Their proof indeed works for arbitrary dominant rational self-maps with minor modifications. For the convenience of the readers, we provide a proof here in the general case. Our proof is based on the proof of  
\cite[Proposition 5.2]{Xie2015}.
\begin{pro}\label{Fakhruddin 5.5}
Let $p>0$ be a prime number.
Let $X$ be a variety over $\overline{\mathbb{F}_p}$. Let $f:X\dashrightarrow X$ be a dominant rational self-map. Then for every Zariski dense open subset $W$ of $X$, $\Per_W(f)(\overline{\mathbb{F}_p})$ is Zariski dense in $X$.
\end{pro}

The key ingredient to prove Proposition \ref{Fakhruddin 5.5} is Hrushovski's twisted Lang-Weil estimate.
\begin{thm}[\cite{hu7,Shuddhodan2022}]\label{hrushovski}
Let $g:X\rightarrow \Spec k$ be an irreducible affine variety of dimension $r$
over an algebraically closed field $k$ of characteristic $p$, and let $q$ be a
power of $p$. We denote by $\phi_q$ the $q$-Frobenius map of $k$, and by $X^{\phi_q}$ the same scheme as $X$ with $g$ replaced by $g\circ \phi_q^{-1}.$
Let $V\subseteq X\times X^{\phi_q}$ be an irreducible subvariety of dimension $r$ such that both projections $$\pi_1:V\rightarrow X \text{ and } \pi_2:V\rightarrow X^{\phi_q}$$ are
dominant and the second one is quasi-finite. Let $\Phi_q \subseteq
X\times X^{\phi_q}$ be the graph of the $q$-Frobenius map $\phi_q$. Set
$$u=\frac{\deg\pi_1}{\deg_{insep}\pi_2},$$ where $\deg\pi_1$ denotes the degree of field extension $K(V)/K(X)$ and $\deg_{insep}\pi_2$ is the purely inseparable degree of the field extension $K(V)/K(X)$.

Then there is a constant
$C$ that does not depend on $q$, such that
$$|\#(V \bigcap \Phi_q )-uq^r|\leq Cq^{r-1/2}.$$
\end{thm}
Hrushovski's original proof of Theorem \ref{hrushovski} relies on model theory. See \cite{Shuddhodan2022} for an algebro-geometric proof.

\medskip

\begin{proof}[Proof of Proposition \ref{Fakhruddin 5.5}]
After replacing $X,f$ by $W, f|_W$, we may assume that $W=X.$
Let $Z:=\overline{\Per(f)(\overline{\mathbb{F}_p})}$ and assume by contradiction that $Z\neq X$. Set $Y:=Z\bigcup I(f). $ Then $Y$ is a proper closed subset of $X$. Let $q=p^n$ be such that $X$ and $f$ are defined
over the subfield $\mathbb{F}_q$ of $\overline{\mathbb{F}_p}$ having exactly $q$
elements. Let $\phi_q$ denote the Frobenius morphism acting on $X$ and let
$\Gamma_{f}$(resp. $\Gamma_m$) denote the graph of $f$ (resp.
$\phi_q^m$) in $X\times X$. Let $U$ be an irreducible affine open
subset of $X\setminus Y$ that is also defined over $\mathbb{F}_q$ and such that $f$ is
an open embedding from $U$ to $X$. Set
$V=\Gamma_{f}\bigcap(U\times U)$. By Theorem \ref{hrushovski},
there exists an integer $m>0$ such that
(V$\bigcap$$\Gamma_m$)($\overline{\mathbb{F}_p}$)$\neq$$\emptyset$ i.e. there
exists $u\in U(\overline{\mathbb{F}_p})$ such that $f(u)=\phi_q^m(u)\in U$.
Since $f$ is defined over $\mathbb{F}_q$, it follows that $f^l(u)=\phi_q^{lm}(u)\in U$ for all $l\geq 0$. This
contradicts the definition of $Y$ and $U$.
\end{proof}

\subsection{Isolated periodic points}
A periodic point $x\in \Per(f)(\overline{\bk})$ is called \emph{isolated} if it is isolated in $\Per_r(f)$ for some period $r\geq 1$ of $x$.
The following result shows we can lift isolated periodic points from the special fiber to the generic fiber.
This result was originally proved by Fakhruddin and Poonen \cite[Theorem 5.1]{fa} for endomorphisms. However, its proof works for arbitrary dominant rational self-maps with minor modifications. For the convenience of the reader, we provide a proof here in the general case. Our proof is based on the proof of  
\cite[Proposition 5.4]{Xie2015}.

\begin{lem}\label{dvr}
Let $\sX$ be a quasi-projective scheme, flat over a discrete valuation ring $R$ with fraction field $K$ and residue field $k_p$. Let $f_R$  be a dominant rational self-map $\sX\dashrightarrow \sX$ over $R$. Let $X_p$ be the
special fiber of $\sX$ and $X$ be the generic fiber of
$\sX$. Assume that $X_p$ is reduced and $X_p\not\subseteq I(f_R).$
Let $f$ be the restriction of $f_R$ to $X$, and $f_p$ be the
restriction of $f_R$ to $X_p$.  Let $U_p$ be a Zariski dense open subset of $X_p$ such that $U_p\cap I(f_R)=\emptyset.$
Let $r\geq 1$ and $x_p\in U_p$ be a closed point in $\Per_{U_p}(f_p)$ of period $r$. Assume that $X_p$ is regular at $x_p$ and 
$x_p$ is isolated in $\Per_r(f_p).$
Then there is a closed isolated periodic point $x\in \Per_r(X)$ such that $x_p\in \overline{\{x\}}.$

Moreover, if isolated closed periodic points in $\Per_{U_p}(X_p)$ are Zariski dense in $X_p$, then the set of isolated $f$-periodic points is Zariski dense in $X$.
\end{lem}

\begin{proof}
 The set of periodic $\overline{k_p}$-points of $f_p$ of period 
$n$ can be viewed as the set of $\overline{k_p}$-points in
$\Delta_{X_p}\bigcap\Gamma_{f_p^n}$, where $\Delta_{X_p}$ is the diagonal
and $\Gamma_{f_p^n}$ is the graph of $f_p^n$ in $X_p\times X_p$.

For any positive integer $r\geq 1$, consider the subscheme
$\Delta_{\sX}\bigcap\Gamma_{f_R^r}$ of
$\sX\times_R\sX$, where $\Delta_{\sX}$ is the
diagonal and $\Gamma_{f_R^r}$ is the graph of $f_R^r$ in
$\sX\times_R \sX$. 
Note that $(x_p,x_p)\subset \Delta_{\sX}\bigcap\Gamma_{f_R^r}.$
As $\sX\times_R \sX$ is regular at $(x_p,x_p)$, $\dim_{(x_p,x_p)}\Delta_{\sX}\bigcap\Gamma_{f_R^r}\geq 1.$ 
As $U_p\cap I(f_R)=\emptyset$, and $x_p\in \Per_{U_p}(f_p)$,
$x_p\not\in I(f_R^r)$, $\Delta_{X_p}\bigcap\Gamma_{f_p^n}$ and the special fiber of  $\Delta_{\sX}\bigcap\Gamma_{f_R^r}$ are locally the same at $(x_p,x_p).$
As $x_p$ is isolated in $\Delta_{X_p}\bigcap\Gamma_{f_p^n}$, we have $\dim_{(x_p,x_p)}\Delta_{\sX}\bigcap\Gamma_{f_R^r}=1$ and every irreducible component of $\Delta_{\sX}\bigcap\Gamma_{f_R^r}$ passing through $(x_p,x_p)$ dominates $\Spec R.$ Pick an irreducible component $\sV$ of $\Delta_{\sX}\bigcap\Gamma_{f_R^r}$
passing through $(x_p,x_p)$. Let $x'$ be the generic point of $\sV$. Then $x'\subseteq \Delta_{X}\bigcap\Gamma_{f^n}$, where $\Delta_{X}$ is the diagonal
and $\Gamma_{f^n}$ is the graph of $f^n$ in $X\times X$. Identify $\Delta_{X}$ with $X$, we get a closed isolated periodic point $x\in \Per_r(X)$ such that $x_p\in \overline{\{x\}}.$

Now assume that isolated closed periodic points in $\Per_{U_p}(X_p)$ are Zariski dense in $X_p$.
We identify $\sX$ with $\Delta_{\sX}$. For any open
subset $U'$ of $X$, let $Z$ be a Cartier divisor of $X$ containing
$X\setminus U'$. Let $\sZ$ be the closure of $Z$ in $\sX$,
then codim$(\sZ)=1$ and every component of $\sZ$ meets
$Z$.  Every irreducible component of $X_p$ is of codimension $1$. If
$X_p\subseteq \sZ$, every irreducible component of $X_p$ is a component of
$\sZ$. Since $X\bigcap X_p=\emptyset$, we get
$X_p\not\subseteq\sZ$. Let $\sV=\sX\setminus\sZ$ and
$V_p=\sV\bigcap X_p$, then $\sV\bigcap X=U'$ and
$V_p\neq\emptyset$. Hence $V_p\cap U_p\neq\emptyset.$
As $X_p$ is reduced, by our assumption,
there is a closed isolated periodic point $x_p$ of period $r\geq 1$ in $\Per_{U_p}(X_p)\cap V_p$ such that $\sX$ is regular at $x_p.$
The previous paragraph shows that there is a closed isolated periodic point $x\in \Per_r(X)$ such that $x_p\in \overline{\{x\}}.$
Then $x\in U'$, this concludes the proof.
\end{proof}

\medskip

Next we show that for cohomologically hyperbolic self-maps, periodic points under mild conditions are isolated.
%To simplify the notations, we write $$\la_j:=\la_j(f), \mu_j:=\mu_j(f), \deg_jf^n:= \deg_{j,L}f^n, \text{ and } L_n:=(f^n)^*L$$ for every $n\geq 0.$

\begin{lem}\label{lemcurvedegregrowth}
Assume that $X$ is projective. Let $L$ be an ample line bundle on $X.$
If $f$ is $i$-cohomologically hyperbolic, then for every $\beta<\mu_i$ there is an affine Zariski open subset $U$ of $X$ such that for every irreducible curve $C$, if $C\cap U_f\neq \emptyset$ and $\dim f^n(C)=1$ for all $n\geq 0$,
we have $$\liminf_{n\to \infty}(L_n\cdot C)^{1/n}\geq \beta.$$
\end{lem}

Recall that the above intersections $(L_n\cdot C)$ are well-defined as in the last paragraph of Section \ref{subseclinebun}.

\proof
To simplify the notations, we write $$\la_j:=\la_j(f), \mu_j:=\mu_j(f), \text{ and } L_n:=(f^n)^*L$$ for every $n\geq 0.$

We may assume that $\beta>1.$
Pick $\epsilon\in (0,1)$, such that $\mu_i\epsilon^2>\beta$ and $\mu_{i+1}\epsilon^{-2}<1.$
There is $m_0\geq 1$ such that for every $m\geq m_0$, we have 
\begin{equation}\label{equationsumuiihyp}\mu_i^{m}\epsilon^{2m}+\mu_{i+1}^m\epsilon^{-2m}\leq \mu_i^m\epsilon^m.
\end{equation}

By Theorem \ref{thmmaincompudegone} or \cite[Proposition 3.5]{Matsuzawa},  there is $m_{1}>m_0$, such that for every $m\geq m_{1},$
$$M_m:=L_{2m}+\mu_i^{m}\mu_{i+1}^mL-\epsilon^m\mu_i^{m}L_m$$ is big.
Fix $m\geq m_1.$
Set $U:=X\setminus \bB_X(M_m).$
Let $C$ be an irreducible curve satisfying  $C\cap U_f\neq \emptyset.$
Then for every $n\geq 0$, $f^n(C)\cap U\neq \emptyset.$
Then $(M_m\cdot f_*^n(C))$ is well-defined and non-negative. 
So we get 
$$(L_{(2+n)m}\cdot C)+\mu_i^{m}\mu_{i+1}^m(L_{nm}\cdot C)\geq \epsilon^m\mu_i^{m}(L_{(n+1)m}\cdot C).$$
As $\dim f^n(C)=1$ for all $n\geq 0$, $(L_n\cdot C)=(L\cdot (f^n)_*C)\geq 1$ (c.f. (\ref{equationprofubd}) of Section \ref{subseclinebun}).
As $\mu_{i+1}^m\epsilon^{-2m}<1$, there is $N\geq 0$ such that 
$$(L_{m(N+1)}\cdot C)>\mu_{i+1}^m\epsilon^{-2m}(L_{mN}\cdot C).$$
By (\ref{equationsumuiihyp}) and Lemma \ref{lemrecuineq}, we have 
$$\liminf_{n\to \infty}(L_n\cdot C)^{1/n}\geq \beta.$$
This concludes the proof.
\endproof

\begin{cor}\label{corisoperiodiccohhyp}
Let $X$ be a variety over $\bk$. Let $f: X\dashrightarrow X$ be a dominant rational self-map which is cohomologically hyperbolic.
Then there is a Zariski dense open subset $U$ of $X$ such that for every $x\in \Per_U(f)$, $x$ is isolate in $\Per_r(f)$, where $r\geq 1$ is a period of $x.$
\end{cor}
\proof
If Corollary \ref{corisoperiodiccohhyp} holds for one Zariski dense open subset $U$, it holds for any Zariski dense open subset $U'$ of $U.$

After replace $X$ by a Zariski dense affine open subset $X'$ and $f$ by $f|_{X'}$, we may assume that $X$ is quasi-projective.
Pick a projective compactification $X''$ of $X$. Then $f$ extends to an dominant rational self-map $f''$ on $X''.$
After replace $X,f$ by $X'',f''$, we may assume that $X$ is projective. Let $L$ be an ample line bundle on $X.$

Let $U$ as in Lemma \ref{lemcurvedegregrowth}.
Let $x\in \Per_U(f)$ of period $r\geq 1.$ If it is not isolated in $\Per_r(f)$ , then there is an irreducible curve $C$ containing $x$ such that $f^r|_C=\id.$ It follows that for every $n\geq 0$,
$$((f^{rn})^*L\cdot C)=(L\cdot C),$$
which is a contradiction.
\endproof

\subsection{Periodic points of cohomologically hyperbolic self-maps}
\proof[Proof of Theorem \ref{thmperiodisodense}]
As we may replace $X,f$ by $V, f|_V$, we only need to prove the case where $X=V.$

Assume that $f$ is $i$-cohomologically hyperbolic for some $i\geq 0.$
After base change by $\overline{\bk}$, we may assume that $\bk$ is algebraically closed.
We may assume that the transcendence
degree of $\mathbf{k}$ over its prime field $F$ is finite, since we can
find a subfield of $\mathbf{k}$ which is finitely generated over $F$ such that $X$ and $f$ are all defined over this subfield.
We complete the proof by induction on the transcendence degree of $\mathbf{k}$ over $F$.

\medskip

If $\bk$ is the closure of a finite field, we conclude the proof by Proposition \ref{Fakhruddin 5.5}.

If $\mathbf{k}=\overline{\mathbb{Q}}$, there is a regular subring $R$ of
$\overline{\mathbb{Q}}$ which is finitely generated over
$\mathbb{Z}$, such that $X$ and $f$ are defined over $R$ i.e.
there is a flat $R$-scheme $\pi: \sX\to \Spec R$, a dominant rational self-map $f_R: \sX\dashrightarrow \sX$ such that 
$X,f$ are the generic fiber of $\sX,f_R$.
After shrinking $\Spec R$, we may assume that for every point $p\in \Spec R$, the fiber $X_p$ is reduced and irreducible, $X_p\not\subseteq I(f_R)$ and  $f_p:=f_R|_{X_p}$ is a dominant rational map.
By Theorem \ref{thmlscdynadeg}, there is a closed point $p\in \Spec R$ such that $f_p$ is $i$-cohomologically hyperbolic. Since $R$ is regular
and $\Frac(R)$ is a number field, the
localization $R_{p}$ of $R$ at $p$ is a discrete valuation ring such that
$\overline{\Frac(R_{p})}=\overline{\mathbb{Q}}$. So
$R_{p}/pR_{p}=R/p$ is a finite field. 
Let $U_p$ be a Zariski dense open subset of $X_p$ with $I(f_R)\cap U_p=\emptyset.$ By Corollary \ref{corisoperiodiccohhyp},
after shrinking $U_p$, we may assume that every periodic point in  $\Per_{U_p}(f_p)$ are isolated.
By Proposition \ref{Fakhruddin 5.5}, $\Per_{U_p}(f_p)$ is Zariski dense. We then conclude the proof by Lemma \ref{dvr}.

If the transcendence degree of $\mathbf{k}$ over $F$ is greater than $1$, we pick an algebraically closed subfield $K$
of $\mathbf{k}$ such that the transcendence degree of $K$ over $F$ equals the transcendence degree of $\mathbf{k}$ over $F$ minus 1. Then we pick a subring $R$
of $\mathbf{k}$ which is finitely generated over $K$, such that $X$ and $f$ are all
defined over $R$. Since $\Spec R$ is regular on an open set, we may assume that $R$ is regular by adding finitely many inverses of elements in $R$. We may repeat the same arguments as in the case $\mathbf{k}=\overline{\mathbb{Q}}$ to conclude the proof.
\endproof

In the end, we give examples to show that for cohomologically non-hyperbolic maps, one can not determine whether the set of periodic points are Zariski dense from the dynamical degrees. 
\subsection{Examples of cohomologically non-hyperbolic maps}\label{subsectionnonhyp}
Let $X$ be a projective variety over $\bk$ of dimension $d\geq 1$. Let $i=1,\dots, d$ and $f: X\dashrightarrow X$ be a $i$-cohomologically hyperbolic maps.
Let $g:\P^N\to \P^N$ be an automorphism over $\bk$.
Consider the rational self-map $F:=f\times g: X\times \P^N\dashrightarrow X\times \P^N.$

\medskip
The product formula for relative dynamical degrees (c.f. \cite{Dinh2011}, \cite{Dang2020} and \cite[Theorem 1.3]{Truong2020}) shows the following lemma.
\begin{lem}\label{lemreldydegf}
We have 
$$\la_j(F)=\la_j(f) \text{ for } j\leq i;$$
$$\la_j(F)=\la_i(f) \text{ for } i\leq j\leq i+N;$$
$$\la_j(F)=\la_{j-N}(f) \text{ for } j\geq i+N.$$
In particular, the dynamical degrees of $F$ does not depend on $g$.
\end{lem}

We note that $\Per(g)$ is Zariski dense if and only if  $g$ is of finite order i.e. $g^m=\id$ for some $m\geq 1.$
Combining this fact with Theorem \ref{thmperiodisodense}, we get the following statement.
\begin{lem}\label{lemfperdensecon}
The set of periodic points of $F$ is Zariski dense if and only if $g$ is of finite order i.e. $g^m=\id$ for some $m\geq 1.$
\end{lem}

\section{Applications to the Kawaguchi-Silverman conjecture}\label{sectionksc}

The aim of this section is to prove the Kawaguchi-Silverman conjecture for certain rational self-maps on projective surfaces. 
In particular, our result implies the Kawaguchi-Silverman conjecture for birational self-maps on projective surfaces. 
We first recall the arithmetic degree and the Kawaguchi-Silverman conjecture.

\subsection{Arithmetic degree}\label{subsecarithdeg}
The arithmetic degree was first defined in \cite{Kawaguchi2016} over a number field or a function field of characteristic zero.
As in \cite{Xie2023, Xie2023a} and \cite[Remark 1.14]{Matsuzawa2020a}, this definition can be extended to characteristic positive.
Here we only recall the definition in the number fields cases.
%Here we follows the way of \cite{Xie2023, Xie2023a} to define the arithmetic degree in the general case.

\medskip

Let $X$ be a projective variety over $\overline{\Q}.$ 
For every $L\in \Pic(X)$, we denote by $h_L: X(\overline{\Q})\to \R$ a Weil height associated to $L$. It is unique up to adding a bounded function.

%
%\medskip
%
%
%As in \cite{Jia2021,Xie2023,Xie2023a}, we define an \textit{admissible triple} to be $(X,f,x)$ where $X$ is a quasi-projective variety over $\bk$, $f\colon X\dashrightarrow X$ is a dominant rational self-map and $x\in X_f(\bk)$.
%
%We say that $(X,f,x)$ \textit{dominates} (resp.~\textit{generically finitely dominates}) $(Y,g,y)$ if there is a dominant rational map (resp.~generically finite and dominant rational map) $\pi\colon X\dashrightarrow Y$ such $\pi\circ f=g\circ\pi$, $\pi$ is well defined along $O_f(x)$ and $\pi(x)=y$.
%
%
%We say that $(X,f,x)$ is \textit{birational} to $(Y,g,y)$ if there is a birational map $\pi\colon X\dashrightarrow Y$ such $\pi\circ f=g\circ\pi$ and if there is a Zariski dense open subset $V$ of $Y$ containing $O_g(y)$ such that $\pi|_U: U:=\pi^{-1}(V)\to V$ is a well-defined isomorphism and $\pi(x)=y$.
%In particular, if $(X,f,x)$ is birational to $(Y,g,y)$, then $(X,f,x)$ generically finitely dominates $(Y,g,y)$.
%
%
%\begin{rem}
%	\leavevmode
%	\begin{enumerate}
%		\item If $(X,f,x)$ dominates $(Y,g,y)$ and if $O_f(x)$ is Zariski dense in $X$, then $O_g(y)$ is Zariski dense in $Y$.
%		Moreover, if $(X,f,x)$ generically finitely dominates $(Y,g,y)$, then $O_f(x)$ is Zariski dense in $X$ if and only if $O_g(y)$ is Zariski dense in $Y$.
%		\item Every admissible triple $(X,f,x)$ is birational to an admissible triple $(X',f',x')$ where $X'$ is projective.
%		Indeed, we may pick $X'$ to be any projective compactification of $X$, $f'$ the self-map of $X'$ induced from $f$, and $x'=x$.
%	\end{enumerate}
%\end{rem}

Let $f\colon X\dashrightarrow X$ is a dominant rational self-map and $x\in X_f(\bk)$.
As in \cite{Jia2021,Xie2023,Xie2023a}, 
we will associate to $(X,f,x)$ a subset $$A_f(x)\subseteq [1,\infty]$$ as follows:
Let $L$ be an ample divisor on $X$, define
$$A_f(x):=\cap_{m\geq 0}\overline{\{(h_L^+(f^n(x)))^{1/n}|\,\, n\geq m\}}\subseteq [1,\infty]$$
to be the limit set  of the sequence $(h_L^+(f^n(x)))^{1/n}$, $n\geq 0$, where $h_L^+(\cdot):=\max\{h_L(\cdot),1\}$.
Indeed we have $A_f(x)\subseteq [1,\la_1(f)]$ by \cite{Kawaguchi2016,Matsuzawa2020a,Jia2021,Xie2023,Song2023, Xie2023a}.
The following lemma shows that the set $A_f(x)$ does not depend on the choice of $L$.

\begin{lemma}\cite[Lemma 2.7]{Xie2023}\label{lemsingwilldef}
	Let $\pi\colon X\dashrightarrow Y$ be a dominant rational map between projective varieties.
	Let $U$ be a Zariski dense open subset of $X$ such that $\pi|_U\colon U\to Y$ is well-defined.
	%Let $V$ be a Zariski dense open subset of $Y$ such that $\pi|_U\colon U:=\pi^{-1}(V)\to V$ is well-defined.
	Let $L$ be an ample divisor on $X$ and $M$ an ample divisor on $Y$.
	Then there are constants $C\geq 1$ and $D>0$ such that for every $x\in U$, we have
	\begin{equation}\label{equationdomineq1}
		h_M(\pi(x))\leq Ch_L(x)+D.
	\end{equation}
	
	Moreover if $V:=\pi(U)$ is open in $Y$ and $\pi|_U\colon U\to V$ is an isomorphism, then 
	there are constants $C\geq 1$ and $D>0$ such that for every $x\in U$, we have
	\begin{equation}\label{equationbirdomineq}
		C^{-1}h_L(x)-D\leq h_M(\pi(x))\leq Ch_L(x)+D.
	\end{equation}
\end{lemma}

%\medskip
%
%Now for every admissible triple $(X,f,x)$, we define $A_f(x)$ to be $A_{f'}(x')$ where $(X',f',x')$ is an admissible triple which is birational to $(X,f,x)$ such that $X'$ is projective.
%By Lemma~\ref{lemsingwilldef}, this definition does not depend on the choice of $(X',f',x')$.
%

\medskip

As in \cite{Kawaguchi2016},
define
\[
\overline{\alpha}_f(x):=\sup A_f(x),\qquad\underline{\alpha}_f(x):=\inf A_f(x),
\]
and call them \emph{upper/lower arithmetic degree}.
%By Lemma~\ref{lemsingwilldef}, if $(X,f,x)$ dominates $(Y,g,y)$, then $\overline{\alpha}_f(x)\geq \overline{\alpha}_g(y)$ and $\underline{\alpha}_f(x)\geq\underline{\alpha}_g(y)$.
By Lemma \ref{lemsingwilldef}, we have the following basic properties:
\begin{pro}\cite[Proposition 6.4]{Xie2023a}\label{probasicaf} We have:
	\begin{enumerate}
		\item $A_f(x)=A_f(f^{\ell}(x))$, for any $\ell\geq 0$.
		\item \label{eq:alpha_pow}
		$A_{f}(x)=\bigcup_{i=0}^{\ell-1}(A_{f^{\ell}}(f^i(x)))^{1/\ell}$.
		In particular, $\overline{\alpha}_{f^{\ell}}(x)=\overline{\alpha}_{f}(x)^{\ell}$, $\underline{\alpha}_{f^{\ell}}(x)=\underline{\alpha}_{f}(x)^{\ell}$.
	\end{enumerate}
\end{pro}

The following result is the Kawaguchi-Silverman-Matsuzawa's upper bound.
See  \cite{Kawaguchi2016,Matsuzawa2020a,Jia2021,Xie2023,Song2023, Xie2023a} for its proof.
\begin{thm}\label{thmuniKSMineq}
Let $h$ be any Weil height on $X$ associated to some ample line bundle.
Then for any $x\in X_f(\bk)$, we have $$\overline{\alpha}_f(x)\leq \la_1(f).$$
\end{thm}

If $\overline{\alpha}_f(x)=\underline{\alpha}_f(x)$,
we set
\[
\alpha_f(x):=\overline{\alpha}_f(x)=\underline{\alpha}_f(x).
\]
In this case, we say that $\alpha_f(x)$ is well-defined and call it the \textit{arithmetic degree} of $f$ at $x$.

\subsection{Kawaguchi-Silverman conjecture}
The following conjecture was proposed by Kawaguchi and Silverman \cite{Silverman2014,Kawaguchi2016}.  
\begin{con}[Kawaguchi-Silverman conjecture]\label{ksc}
Let $X$ be a projective variety over $\overline{\Q}$.
Let $f: X\dashrightarrow X$ be a dominant rational map. Then for every  $x\in X_f(\overline{\Q})$, $\alpha_f(x)$ is well defined.
Moreover, if $O_f(x)$ is Zariski dense, then we have $$\alpha_f(x)= \la_1(f).$$
\end{con}

\subsection{Our result}
The following is the main result of this section.
\begin{thm}[=Theorem \ref{KSCsurfacelaonelain}]\label{KSCsurfacelaonela}Let $X$ be a projective surface over $\overline{\Q}$ and $f: X\dashrightarrow X$ be a dominant rational self-map such that $\la_1(f)>\la_2(f)$ or $\la_2(f)=\la_1(f)^2$. Let $x\in X_{f}(\overline{\Q})$. If the orbit $O_f(x)$ of $x$ is Zariski dense, then $\alpha_f(x)=\la_1(f).$
\end{thm}

In particular, Theorem \ref{KSCsurfacelaonela} implies the Kawaguchi-Silverman conjecture for birational self-maps on projective surfaces.

\proof
Let $L$ be an ample line bundle on $X$. 
To simplify the notations, we write $$\la_j:=\la_j(f), \mu_j:=\mu_j(f), \text{ and } L_n:=(f^n)^*L$$ for every $n\geq 0.$
Let $i:=1$ if $\la_1>\la_2$ and $i:=2$ if $\la_1^2=\la_2.$
Then we have $\mu_i=\la_1>1$ and $\mu_{i+1}<1.$
If $\la_1=1$, Theorem \ref{KSCsurfacelaonela} trivially holds. So we assume that $\la_1>1.$

We denote by $h: X(\bk)\to \R$ a Weil height associated to $L$. It is unique up to adding a bounded function. We may assume that $h(y)\geq 1$ for every $y\in X(\overline{\Q})$.

Let $x$ be a point in $X_f(\overline{\Q})$ whose orbit is Zariski dense. By Theorem \ref{thmuniKSMineq}, $\overline{\alpha}_f(x)\leq \la_1.$ We only need to show that for every $\beta\in (0,\la_1)$, $\underline{\alpha}_f(x)\geq  \beta.$
Pick $\epsilon\in (0,1)$ such that 
\begin{equation}\label{equationconappr}
\epsilon^{2}\mu_{i}>\beta \text{ and } \epsilon^{-2}\mu_{i+1}<1.
\end{equation}
There is $m_0\geq 1$ such that  for every $m\geq m_0$, we have 
\begin{equation}\label{equationsumrigcon}(\epsilon^{2}\mu_{i})^m+(\epsilon^{-2}\mu_{i+1})^m<\epsilon^m\mu_{i}^m-1.
\end{equation}
Set $$\beta_1:=(\epsilon^{2}\mu_{i})^m \text{ and } \beta_2:=(\epsilon^{-2}\mu_{i+1})^m.$$
By Theorem \ref{thmmaincompudegone}, there is $m\geq m_0$ such that 
$$M_m:=L_{2m}+\mu_i^m\mu_{i+1}^mL- \epsilon^m\mu_{i}^mL_m$$
is big.
There is a constant $C>0$ such that for every $y\in X_f(\bk)\setminus \bB_{X}(M_m)$, we have 
$$h(f^{2m}(y))+\mu_i^m\mu_{i+1}^mh(y)\geq \epsilon^m\mu_{i}^mh(f^m(y))-C.$$
By the Northcott property, after replacing $x$ by some $f^l(x)$ for some $l\geq 0$, we may assume that $h(f^n(x))\geq C$ for every $n\geq 0$.
Set $h_n:=h(f^{mn}(x)).$ 
We only need to show that 
\begin{equation}\label{equliminfhnbe}\liminf_{n\to \infty} h_n^{1/n}\geq \beta^m.
\end{equation}
By Lemma \ref{lemsingwilldef}, there is a constant $D>1$ such that for every $n\geq 0$,
\begin{equation}\label{equtrivialbound}
h_{n+1}\leq Dh_n.
\end{equation}
If $f^{mn}(x)\not\in \bB_{X}(M_m)$, we have 
\begin{equation}\label{equareingenr}h_{n+2}+\mu_i^m\mu_{i+1}^mh_n\geq \epsilon^m\mu_{i}^mh_{n+1}-C\geq (\epsilon^m\mu_{i}^m-1)h_{n+1}.
\end{equation}
Hence we have 
\begin{equation}\label{equareingenrsimple}h_{n+2}-\beta_2h_{n+1}\geq \beta_1(h_{n+1}-\beta_2h_{n}).
\end{equation}

Let $B_0$ be the union of irreducible components of $\bB_{X}(M_m)$ of dimension $0$ and $B_1$ be the union of irreducible components of $\bB_{X}(M_m)$ of dimension $1$.
We have $\bB_{X}(M_m)=B_0\sqcup B_1.$
After replacing $x$ by some $f^l(x)$ for some $l\geq 0$, we may assume that $f^n(x)\not\in B_0$ for every $n\geq 0.$

Let $C_j, j=1,\dots, s'$ be the irreducible components of $B_1$ such that $C_j\cap X_f\neq\emptyset.$ 
As $f^n(x)\in X_{f}(\overline{\Q})$ for every $n\geq 0$, 
$f^n(x)\in \bB_{X}(M_m)$ if and only if  $f^n(x)\in C_j$ for some $j=1,\dots,s'.$
We may assume that $C_j, j=1,\dots,s''$ are exactly the $C_j$ such that for every $n\geq 0$, $\dim(f^n(C_j))=1.$
After replacing $x$ by some $f^l(x)$ for some $l\geq 0$, we may assume that $f^n(x)\not\in \cup_{j=s''+1}^{s'} C_j$ for every $n\geq 0.$
We may assume that $C_j, j=1,\dots,s$ are exactly the $C_j, j=1,\dots, s''$ which are not preperiodic.
As $O_f(x)$ is Zariski dense, 
$f^n(x)\in \bB_{X}(M_m)$ if and only if  $f^n(x)\in C_j$ for some $j=1,\dots,s.$

\begin{lem}\label{lemcurvedegreenotbounded}Let $C$ be an irreducible curve in $X$ with $C\cap X_f\neq \emptyset.$ Assume that $\dim f^n(C)=1$ for every $n\geq 0$ and 
$C$ is not preperiodic. Recall that the above intersections $(L_n\cdot C)$ are well-defined as in the last paragraph of Section \ref{subseclinebun}.
Then the sequence $(L_n\cdot C), n\geq 0$ is not bounded. 
\end{lem}
Applying Lemma \ref{lemcurvedegreenotbounded} for $f^m$ and  $f(C_j), j=1,\dots,s$, for every $j=1,\dots,s$, there is $N_j\geq 1$ such that $(L_{N_jm}\cdot f(C_j))>3D^3(L\cdot f(C_j)).$
As $f(C_j)$ is one dimensional, there is $C'>0$ such that for every $j=1,\dots,s$,  for every $y\in X_f(\bk)\cap f(C_j)$
we have $h(f^{N_jm}(y))>3D^3h(y)-C'.$
By the Northcott property, after replacing $x$ by $f^l(x)$ for some $l\geq 0$, we may assume that $h(f^n(x))> C'$ for every $n\geq 0.$
Moreover, we may assume that $h(f^m(x))>h(x).$
If $f^{mn}(x)\in f(C_j), j=1,\dots, s$, we have 
\begin{equation}\label{equationoncjhin}h(f^{(N_j+n)m}(x))>3D^3h(f^{mn}(x))-C'>2D^3h(f^{mn}(x)).
\end{equation}
By (\ref{equtrivialbound}), there is $t_n\in \{0,\dots, N_j-1\}$ such that 
$$h(f^{(t_n+1+n)m}(x))>h(f^{(t_n+n)m}(x))>h(f^{nm}(x)).$$
Set $N:=\max\{N_j|\,\, j=1,\dots,s\}.$ The above discussion shows the following:
 If $f^{m(n-1)}(x)\in \bB_{X}(M_m)$, there is $t_n\in \{0,\dots, N\}$, such that 
 \begin{equation}\label{equaineqinbm}
 h_{n+t_n+1}>h_{n+t_n}>h_n.
 \end{equation}
Set $W:=\{n\geq 1|\,\, f^{m(n-1)}(x)\in \bB_{X}(M_m)\}.$ By the Weak dynamical Mordell-Lang \cite[Corollary 1.5]{Bell2015} (see also \cite[Theorem 2.5.8]{Favre2000a}, \cite[Theorem D, Theorem E]{Gignac2014},\cite[Theorem 2]{Petsche2015}, \cite[Theorem 1.10]{Bell2020}, \cite[Theorem 1.17]{Xie2023} and \cite[Theorem 5.2]{Xie2023a}), we have 
\begin{equation}\label{equationwdml}\lim_{n\to \infty}\frac{w_n}{n}=0
\end{equation}
where $w_n:=\#(\{1,\dots, n\}\cap W)$.
We define a sequence $p(n)$ by induction.
Define $p(0)=0.$ Assume that $p(n)$ is defined for $n\leq n_1$. Define $p(n_1+1)=p(n_1)+1$ if $p(n_1)\not\in W$; otherwise if $n_1\not\in W$, define $p(n_1+1)=p(n_1)+t_{p(n_1)}+1.$
It is clear that $p(n)$ is strictly increasing. 
As $t_{p(n')}\leq N$ for every $n'\geq 0$, for every $n\geq 0$, we have
\begin{equation}\label{equrnncompp}p(n)\geq n \text{ and } p(n+1)\leq p(n)+N.
\end{equation}
For every $n\geq 0$, there is a minimal $r(n)\geq 0$ such that $n\leq p(r(n)).$ It is clear that 
\begin{equation}\label{equrnncomp}r(n)\leq n \text{ and } p(r(n))\leq n+N.
\end{equation}
\begin{lem}\label{lemrnono}We have $$\lim_{n\to \infty}\frac{r(n)}{n}=1.$$
\end{lem}
We note that $r(p(n))=n$. By Lemma \ref{lemrnono}, we have 
\begin{equation}\label{equpnnone}
\lim_{n\to \infty}\frac{p(n)}{n}=1.
\end{equation}

\medskip

The definition of $p(n)$ shows that for $n\geq 1$, if $p(n-1)\in W$, then $h_{p(n)}> h_{p(n)-1}$ and $h_{p(n)}> h_{p(n-1)}$.
Moreover, we have $h_1>h_0.$
For a set $I$ of consecutive integers, we say that $I$ is of type 0 if $p(n)\not\in W$ for every $n\in I$ and 
say that $I$ is of type 1 if $p(n)\in W$ for every $n\in I.$
Let $n\geq 1$, write $\{1,\dots,n\}$ as $I_1\sqcup \cdots \sqcup I_s$ where $I_i$ are set of consecutive integers such that
\begin{points}
\item[$\d$] for $i=1,\dots, s-1,$ $\max\{I_{i}\}+1=\min\{I_{i+1}\}$;
\item[$\d$]for every $i=1,\dots, s$, $I_i$ are either of type 0 or of type 1;
\item[$\d$] for $i=1,\dots, s-1$, the type of $I_i$ and $I_{i+1}$ are different.
\end{points}
Write $I_{i}=\{a_i, a_i+1\dots, b_i\}.$  As $0\not\in W$, we have $a_1=1$ and $p(a_1)=1.$ 
If $i\leq s-1$, then $a_{i+1}=b_i+1$. If further that $I_i$ is of type 0, then $p(a_{i+1})=p(b_{i})+1.$

By (\ref{equareingenrsimple}),
 if $I_i$ is of type 0, then for every $j\in p(a_i),\dots, p(b_i)+1$, we have
$$h_{j}\geq \beta_1^{j-a_i}(h_{p(a_i)}-\beta_2h_{p(a_i)-1}).$$
If $i=1$, we have $h_{p(a_i)}=h_1>h_0=h_{p(a_i)-1}.$
If $i\geq 2$, we have $a_i-1=b_{i-1}$ and $I_{i-1}$ is of type 1. As $p(a_i-1)\in W$, we have 
$h_{p(a_i)}>h_{p(a_i)-1}.$ Hence we have 
\begin{equation}\label{equtypzero}
h_{p(b_i)+1}\geq \beta_1^{b_i+1-a_i}(1-\beta_2)h_{p(a_i)}=\beta_1^{\#I_i}(1-\beta_2)h_{p(a_i)}
\end{equation}
and 
\begin{equation}\label{equtypzerolast}
h_{p(b_i)}\geq \beta_1^{b_i-a_i}(1-\beta_2)h_{p(a_i)}=\beta_1^{\#I_i-1}(1-\beta_2)h_{p(a_i)}
\end{equation}
When $i\leq s-1$, $p(b_i)+1=p(a_{i+1}).$

 If $I_i$ is of type 1,  then for every $n\in I_i$, we have $h_{p(n+1)}>h_{p(n)}.$
 So we have 
\begin{equation}\label{equtypone}
h_{p(b_i+1)}\geq h_{p(b_i+1)}\geq h_{p(a_i)}
\end{equation}
When $i\leq s-1$, $p(b_i+1)=p(a_{i+1}).$

Set $O_n:=\{i=1,\dots,s|\,\, I_i \text{ is of type } 0\}.$ Set $e:=0$ if $s\not\in O$ and $e=1$ if $s\in O_n.$
Set $l_n:=\sum_{i\in O_n}\#I_i.$
Combining (\ref{equtypzero}), (\ref{equtypzerolast}) and (\ref{equtypone}), we get 
\begin{equation}\label{equhnlowerb}h_{p(n)}\geq \beta_1^{(\sum_{i\in O_n}\#I_i)-e}(1-\beta_2)^{\#O_n}h_1\geq \beta_1^{l_n-1}(1-\beta_2)^{\#O_n}h_1 .
\end{equation}
Since
$$\#O_n\leq n-l_n=\sum_{i\in \{1,\dots, s\}\setminus O_n}\#I_i\leq w_{p(n)},$$
by (\ref{equationwdml}) and (\ref{equpnnone}), we get that
\begin{equation}\label{equlnnone}
\lim_{n\to \infty}\frac{l_n}{n}=1 \text{ and } \frac{\#O_n}{n}=0.
\end{equation}
Then by (\ref{equhnlowerb}), we have 
\begin{equation}\label{equationliminfhpn}
\liminf_{n\to \infty} h_{p(n)}^{1/n}\geq \beta_1.
\end{equation}

By (\ref{equrnncomp}) and (\ref{equtrivialbound}), for every $n\geq 1$, we have $h_{n}\geq D^{-N}h_{p(r(n))}.$
By (\ref{equationliminfhpn}) and Lemma \ref{lemrnono}, we get that 
$$\liminf_{n\to \infty} h_{n}^{1/n}\geq \liminf_{n\to \infty}((D^{-N}h_{p(r(n))})^{1/r(n)})^{n/r(n)}\geq \beta_1\geq \beta^m.$$
Hence (\ref{equliminfhnbe}) holds. This concludes the proof.
\endproof

\proof[Proof of Lemma \ref{lemcurvedegreenotbounded}]
If $\la_1>\la_2$, then $f$ is cohomologically $1$-hyperbolic.
If $\la_1^2=\la_2$, as $\la_1>1$, $f$ is cohomologically $2$-hyperbolic.
By Lemma \ref{lemcurvedegregrowth}, there is a non-empty Zariski open subset $V$ of $X$ such that for every irreducible curve $C''$ of $X$, if 
$\dim f^n(C'')=1$ for every $n\geq 0$ and 
$C''\cap V_f\neq\emptyset$. Then the sequence $(L_n\cdot C''), n\geq 0$ is not bounded. 

As $C$ is not preperiodic and $\dim X\setminus V\leq 1$, after replacing $C$ by $f^l(C)$ for some $l\geq 0$, we may assume that $f^n(C)\cap V\neq\emptyset.$
Hence the generic point of $C$ is contained in $V_f$.  We conclude the proof by the previous paragraph. 
\endproof

\proof[Proof of Lemma \ref{lemrnono}]
For $j=0,\dots, r(n)-1$, we have 
$$p(j+1)-p(j)=1$$
if $j\not\in W;$
and 
$$p(j+1)-p(j)\leq N$$
if $j\in W.$
Hence we have $$r(n)\leq n\leq p(r(n))\leq r(n)+(N-1)w_{p(r(n))}\leq r(n)+(N-1)w_{n+N}.$$
So we have 
$$\limsup_{n\to \infty} \frac{r_n}{n}\leq 1\leq \liminf_{n\to \infty}\frac{r_n}{n}+(N-1)\lim_{n\to \infty}\frac{w_{n+N}}{n}=\liminf_{n\to \infty}\frac{r_n}{n}.$$
This concludes the proof.
\endproof

	\bibliography{dd}

\newcommand{\etalchar}[1]{$^{#1}$}
\begin{thebibliography}{AAdBM99}

\bibitem[AAdBM99]{Abarenkova1999}
N.~Abarenkova, J.-Ch. Angl\`es~d'Auriac, S.~Boukraa, and J.-M. Maillard.
\newblock Growth-complexity spectrum of some discrete dynamical systems.
\newblock {\em Phys. D}, 130(1-2):27--42, 1999.

\bibitem[AdMV06]{AnglesdAuriac2006}
J.-Ch. Angl\`es~d'Auriac, J.-M. Maillard, and C.~M. Viallet.
\newblock On the complexity of some birational transformations.
\newblock {\em J. Phys. A}, 39(14):3641--3654, 2006.

\bibitem[BC16]{Blanc-Cantat}
J\'er\'emy Blanc and Serge Cantat.
\newblock Dynamical degrees of birational transformations of projective
  surfaces.
\newblock {\em J. Amer. Math. Soc.}, 29(2):415--471, 2016.

\bibitem[BD01]{Briend2001}
J.-Y. Briend and J.~Duval.
\newblock Deux caract\'erisations de la mesure d'\'equilibre d'un endomorphisme
  de $\mathbb{P}^k(\mathbf{C})$.
\newblock {\em Publ. Math. Inst. Hautes \'Etudes Sci.}, 93:145¨C--159, 2001.

\bibitem[BD05]{Bedford2005}
Eric Bedford and Jeffrey Diller.
\newblock Energy and invariant measures for birational surface maps.
\newblock {\em Duke Math. J.}, 128(2):331--368, 2005.

\bibitem[BDJ20]{Bell2020a}
Jason~P. Bell, Jeffrey Diller, and Mattias Jonsson.
\newblock A transcendental dynamical degree.
\newblock {\em Acta Math.}, 225(2):193--225, 2020.

\bibitem[BDJK23]{Bell2023}
Jason Bell, Jeffrey Diller, Mattias Jonsson, and Holly Krieger.
\newblock {B}irational maps with transcendental dynamical degree.
\newblock arXiv:2107.04113, 2023.

\bibitem[BFs00]{Bonifant2000}
Araceli~M. Bonifant and John~Erik Forn\ae~ss.
\newblock Growth of degree for iterates of rational maps in several variables.
\newblock {\em Indiana Univ. Math. J.}, 49(2):751--778, 2000.

\bibitem[BG06]{Bombieri2006}
Enrico Bombieri and Walter Gubler.
\newblock {\em Heights in {D}iophantine geometry}, volume~4 of {\em New
  Mathematical Monographs}.
\newblock Cambridge University Press, Cambridge, 2006.

\bibitem[BGT15]{Bell2015}
Jason~P. Bell, Dragos Ghioca, and Thomas~J. Tucker.
\newblock The dynamical {M}ordell-{L}ang problem for {N}oetherian spaces.
\newblock {\em Funct. Approx. Comment. Math.}, 53(2):313--328, 2015.

\bibitem[BHS20]{Bell2020}
Jason~P. Bell, Fei Hu, and Matthew Satriano.
\newblock Height gap conjectures, {$D$}-finiteness, and a weak dynamical
  {M}ordell-{L}ang conjecture.
\newblock {\em Math. Ann.}, 378(3-4):971--992, 2020.

\bibitem[BIJ{\etalchar{+}}19]{Benedetto2019}
Robert Benedetto, Patrick Ingram, Rafe Jones, Michelle Manes, Joseph~H.
  Silverman, and Thomas~J. Tucker.
\newblock Current trends and open problems in arithmetic dynamics.
\newblock {\em Bull. Amer. Math. Soc. (N.S.)}, 56(4):611--685, 2019.

\bibitem[BK04]{Bedford2004}
Eric Bedford and Kyounghee Kim.
\newblock On the degree growth of birational mappings in higher dimension.
\newblock {\em J. Geom. Anal.}, 14(4):567--596, 2004.

\bibitem[BK08]{Bedford2008}
Eric Bedford and Kyounghee Kim.
\newblock Degree growth of matrix inversion: birational maps of symmetric,
  cyclic matrices.
\newblock {\em Discrete Contin. Dyn. Syst.}, 21(4):977--1013, 2008.

\bibitem[BLS93]{Bedford1993}
E.~Bedford, M.~Lyubich, and J.~Smillie.
\newblock Distribution of periodic points of polynomial diffeomorphisms of
  {$\bold C^2$}.
\newblock {\em Invent. Math.}, 114(2):277--288, 1993.

\bibitem[BV99]{Bellon1999}
M.P. Bellon and C-M. Viallet.
\newblock Algebraic entropy.
\newblock {\em Comm. Math. Phys.}, 204:425--437, 1999.
\newblock chao-dyn/9805006.

\bibitem[Can99]{Cantat1999}
Serge Cantat.
\newblock Dynamique des automorphismes des surfaces projectives complexes.
\newblock {\em C. R. Acad. Sci. Paris S\'{e}r. I Math.}, 328(10):901--906,
  1999.

\bibitem[Can04]{Cantat2004}
Serge Cantat.
\newblock Diff\'{e}omorphismes holomorphes {A}nosov.
\newblock {\em Comment. Math. Helv.}, 79(4):779--797, 2004.

\bibitem[Can11]{Cantat2007}
Serge Cantat.
\newblock Sur les groupes de transformations birationnelles des surfaces.
\newblock {\em Ann. of Math. (2)}, 174(1):299--340, 2011.

\bibitem[CO15]{Cantat2015}
Serge Cantat and Keiji Oguiso.
\newblock Birational automorphism groups and the movable cone theorem for
  {C}alabi-{Y}au manifolds of {W}ehler type via universal {C}oxeter groups.
\newblock {\em Amer. J. Math.}, 137(4):1013--1044, 2015.

\bibitem[Dan20]{Dang2020}
Nguyen-Bac Dang.
\newblock Degrees of iterates of rational maps on normal projective varieties.
\newblock {\em Proc. Lond. Math. Soc. (3)}, 121(5):1268--1310, 2020.

\bibitem[DDG10]{Diller2010}
Jeffrey Diller, Romain Dujardin, and Vincent Guedj.
\newblock Dynamics of meromorphic maps with small topological degree {III}:
  geometric currents and ergodic theory.
\newblock {\em Ann. Sci. \'Ec. Norm. Sup\'er. (4)}, 43(2):235--278, 2010.

\bibitem[Del74]{Deligne1974}
Pierre Deligne.
\newblock La conjecture de {W}eil. {I}.
\newblock {\em Inst. Hautes \'{E}tudes Sci. Publ. Math.}, (43):273--307, 1974.

\bibitem[DF01]{favre}
Jeffrey Diller and Charles Favre.
\newblock Dynamics of bimeromorphic maps of surfaces.
\newblock {\em Amer. J. Math.}, 123(6):1135--1169, 2001.

\bibitem[DF21]{Dang2021}
Nguyen-Bac Dang and Charles Favre.
\newblock Spectral interpretations of dynamical degrees and applications.
\newblock {\em Ann. of Math. (2)}, 194(1):299--359, 2021.

\bibitem[Din05]{Dinh2005c}
Tien-Cuong Dinh.
\newblock Suites d'applications m\'{e}romorphes multivalu\'{e}es et courants
  laminaires.
\newblock {\em J. Geom. Anal.}, 15(2):207--227, 2005.

\bibitem[DN11]{Dinh2011}
Tien-Cuong Dinh and Vi\^{e}t-Anh Nguy\^{e}n.
\newblock Comparison of dynamical degrees for semi-conjugate meromorphic maps.
\newblock {\em Comment. Math. Helv.}, 86(4):817--840, 2011.

\bibitem[Dol18]{Dolgachev2018}
Igor Dolgachev.
\newblock Salem numbers and {E}nriques surfaces.
\newblock {\em Exp. Math.}, 27(3):287--301, 2018.

\bibitem[DS]{Tien-CuongDinh}
Tien-Cuong Dinh and Nessim Sibony.
\newblock Density of positive closed currents and dynamics of {H}enon-type
  automorphisms of $\mathbb{C}^k$ (part {II}).
\newblock arXiv:1403.0070.

\bibitem[DS04]{Dinh2004}
Tien-Cuong Dinh and Nessim Sibony.
\newblock Regularization of currents and entropy.
\newblock {\em Ann. Sci. \'{E}cole Norm. Sup. (4)}, 37(6):959--971, 2004.

\bibitem[DS05]{Dinh2005}
Tien-Cuong Dinh and Nessim Sibony.
\newblock Une borne sup\'erieure pour l'entropie topologique d'une application
  rationnelle.
\newblock {\em Ann. of Math. (2)}, 161(3):1637--1644, 2005.

\bibitem[DS10]{Dinh2010}
Tien-Cuong Dinh and Nessim Sibony.
\newblock Dynamics in several complex variables: endomorphisms of projective
  spaces and polynomial-like mappings.
\newblock In {\em Holomorphic dynamical systems}, volume 1998 of {\em Lecture
  Notes in Math.}, pages 165--294. Springer, Berlin, 2010.

\bibitem[Duj06]{Dujardin2006}
Romain Dujardin.
\newblock Laminar currents and birational dynamics.
\newblock {\em Duke Math. J.}, 131(2):219--247, 2006.

\bibitem[ES13]{Esnault2013}
H\'{e}l\`ene Esnault and Vasudevan Srinivas.
\newblock Algebraic versus topological entropy for surfaces over finite fields.
\newblock {\em Osaka J. Math.}, 50(3):827--846, 2013.

\bibitem[Fak03]{fa}
Najmuddin Fakhruddin.
\newblock Questions on self maps of algebraic varieties.
\newblock {\em J. Ramanujan Math. Soc.}, 18(2):109--122, 2003.

\bibitem[Fav00]{Favre2000a}
Charles Favre.
\newblock {\em Dynamique des applications rationnelles}.
\newblock PhD thesis, 2000.

\bibitem[Fav03]{Favre2003}
Charles Favre.
\newblock Les applications monomiales en deux dimensions.
\newblock {\em Michigan Math. J.}, 51(3):467--475, 2003.

\bibitem[FJ07]{Favre2007}
Charles Favre and Mattias Jonsson.
\newblock Eigenvaluations.
\newblock {\em Ann. Sci. \'Ecole Norm. Sup. (4)}, 40(2):309--349, 2007.

\bibitem[FJ11]{Favre2011}
Charles Favre and Mattias Jonsson.
\newblock Dynamical compactifications of $\bold {C}^2$.
\newblock {\em Ann. of Math. (2)}, 173(1):211--248, 2011.

\bibitem[FRL10]{Favre2010}
Charles Favre and Juan Rivera-Letelier.
\newblock Th\'{e}orie ergodique des fractions rationnelles sur un corps
  ultram\'{e}trique.
\newblock {\em Proc. Lond. Math. Soc. (3)}, 100(1):116--154, 2010.

\bibitem[FTX22]{Favre2022}
Charles Favre, Tuyen~Trung Truong, and Junyi Xie.
\newblock Topological entropy of a rational map over a complete metrized field.
\newblock arXiv:2208.00668, 2022.

\bibitem[Ful84]{Fulton1984}
William Fulton.
\newblock {\em Intersection theory}, volume~2 of {\em Ergebnisse der Mathematik
  und ihrer Grenzgebiete (3) [Results in Mathematics and Related Areas (3)]}.
\newblock Springer-Verlag, Berlin, 1984.

\bibitem[FW12]{Favre2012}
Charles Favre and Elizabeth Wulcan.
\newblock Degree growth of monomial maps and {M}c{M}ullen's polytope algebra.
\newblock {\em Indiana Univ. Math. J.}, 61(2):493--524, 2012.

\bibitem[Ghy95]{Ghys1995}
\'{E}tienne Ghys.
\newblock Holomorphic {A}nosov systems.
\newblock {\em Invent. Math.}, 119(3):585--614, 1995.

\bibitem[Gig14]{Gignac2014}
William Gignac.
\newblock Measures and dynamics on {N}oetherian spaces.
\newblock {\em J. Geom. Anal.}, 24(4):1770--1793, 2014.

\bibitem[Gro64]{EGA-IV-I}
A.~Grothendieck.
\newblock \'{E}l\'{e}ments de g\'{e}om\'{e}trie alg\'{e}brique. {IV}. \'{E}tude
  locale des sch\'{e}mas et des morphismes de sch\'{e}mas. {I}.
\newblock {\em Inst. Hautes \'{E}tudes Sci. Publ. Math.}, (20):259, 1964.

\bibitem[Gro03]{Gromov2003}
Mikha\"{\i}l Gromov.
\newblock On the entropy of holomorphic maps.
\newblock {\em Enseign. Math. (2)}, 49(3-4):217--235, 2003.

\bibitem[GT09]{Ghioca2009}
Dragos Ghioca and Thomas~J. Tucker.
\newblock Periodic points, linearizing maps, and the dynamical {M}ordell-{L}ang
  problem.
\newblock {\em J. Number Theory}, 129(6):1392--1403, 2009.

\bibitem[Gue05a]{Guedj2005a}
Vincent Guedj.
\newblock Entropie topologique des applications m\'{e}romorphes.
\newblock {\em Ergodic Theory Dynam. Systems}, 25(6):1847--1855, 2005.

\bibitem[Gue05b]{Guedj2005}
Vincent Guedj.
\newblock Ergodic properties of rational mappings with large topological
  degree.
\newblock {\em Ann. of Math. (2)}, 161(3):1589--1607, 2005.

\bibitem[Hru]{hu7}
Ehud Hrushovski.
\newblock The elementary theory of the {F}robenius automorphisms.
\newblock arXiv:math/0406514v1.

\bibitem[HS00]{Hindry2000}
Marc Hindry and Joseph~H. Silverman.
\newblock {\em Diophantine geometry}, volume 201 of {\em Graduate Texts in
  Mathematics}.
\newblock Springer-Verlag, New York, 2000.

\bibitem[JL23]{Jiang2023}
Chen Jiang and Zhiyuan Li.
\newblock Algebraic reverse {K}hovanskii-{T}eissier inequality via {O}kounkov
  bodies.
\newblock {\em Math. Z.}, 305(2):Paper No. 26, 14, 2023.

\bibitem[JR18]{Jonsson2018}
Mattias Jonsson and Paul Reschke.
\newblock On the complex dynamics of birational surface maps defined over
  number fields.
\newblock {\em J. Reine Angew. Math.}, 744:275--297, 2018.

\bibitem[JSXZ21]{Jia2021}
Jia Jia, Takahiro Shibata, Junyi Xie, and De-Qi Zhang.
\newblock {E}ndomorphisms of quasi-projective varieties -- towards {Z}ariski
  dense orbit and {K}awaguchi-{S}ilverman conjectures.
\newblock arXiv:2104.05339, 2021.

\bibitem[Kaw06]{Kawaguchi2006}
Shu Kawaguchi.
\newblock Canonical height functions for affine plane automorphisms.
\newblock {\em Math. Ann.}, 335(2):285--310, 2006.

\bibitem[Kaw08]{Kawaguchi2008}
Shu Kawaguchi.
\newblock Projective surface automorphisms of positive topological entropy from
  an arithmetic viewpoint.
\newblock {\em Amer. J. Math.}, 130(1):159--186, 2008.

\bibitem[Kaw13]{Kawaguchi2013}
Shu Kawaguchi.
\newblock Local and global canonical height functions for affine space regular
  automorphisms.
\newblock {\em Algebra Number Theory}, 7(5):1225--1252, 2013.

\bibitem[KS14]{Kawaguchi2014}
Shu Kawaguchi and Joseph~H. Silverman.
\newblock Examples of dynamical degree equals arithmetic degree.
\newblock {\em Michigan Math. J.}, 63(1):41--63, 2014.

\bibitem[KS16]{Kawaguchi2016}
Shu Kawaguchi and Joseph~H. Silverman.
\newblock On the dynamical and arithmetic degrees of rational self-maps of
  algebraic varieties.
\newblock {\em J. Reine Angew. Math.}, 713:21--48, 2016.

\bibitem[Laz04]{Lazarsfeld}
Robert Lazarsfeld.
\newblock {\em Positivity in algebraic geometry. {I}}, volume~48 of {\em
  Ergebnisse der Mathematik und ihrer Grenzgebiete. 3. Folge. A Series of
  Modern Surveys in Mathematics [Results in Mathematics and Related Areas. 3rd
  Series. A Series of Modern Surveys in Mathematics]}.
\newblock Springer-Verlag, Berlin, 2004.
\newblock Classical setting: line bundles and linear series.

\bibitem[Les21]{Lesieutre2021a}
John Lesieutre.
\newblock Tri-{C}oble surfaces and their automorphisms.
\newblock {\em J. Mod. Dyn.}, 17:267--284, 2021.

\bibitem[Lin12]{Lin2012}
Jan-Li Lin.
\newblock Pulling back cohomology classes and dynamical degrees of monomial
  maps.
\newblock {\em Bull. Soc. Math. France}, 140(4):533--549, 2012.

\bibitem[LS23]{Luo2023}
Wenbin Luo and Jiarui Song.
\newblock Arithmetic degrees of dynamical systems over fields of characteristic
  zero.
\newblock arXiv:2401.11982, 2023.

\bibitem[Mat20]{Matsuzawa2020a}
Yohsuke Matsuzawa.
\newblock On upper bounds of arithmetic degrees.
\newblock {\em Amer. J. Math.}, 142(6):1797--1820, 2020.

\bibitem[Mat23]{Matsuzawa2023}
Yohsuke Matsuzawa.
\newblock {R}ecent advances on {K}awaguchi-{S}ilverman conjecture.
\newblock arXiv:2311.15489, 2023.

\bibitem[McM02]{McMullen2002}
Curtis~T. McMullen.
\newblock Dynamics on {$K3$} surfaces: {S}alem numbers and {S}iegel disks.
\newblock {\em J. Reine Angew. Math.}, 545:201--233, 2002.

\bibitem[McM07]{McMullen2007}
Curtis~T. McMullen.
\newblock Dynamics on blowups of the projective plane.
\newblock {\em Publ. Math. Inst. Hautes \'{E}tudes Sci.}, (105):49--89, 2007.

\bibitem[McM11]{McMullen2011}
Curtis~T. McMullen.
\newblock K3 surfaces, entropy and glue.
\newblock {\em J. Reine Angew. Math.}, 658:1--25, 2011.

\bibitem[McM16]{McMullen2016}
Curtis~T. McMullen.
\newblock Automorphisms of projective {K}3 surfaces with minimum entropy.
\newblock {\em Invent. Math.}, 203(1):179--215, 2016.

\bibitem[MHV97]{M.Henneaux1997}
J.~Krasil'shchik M.~Henneaux and A.~Vinogradov, editors.
\newblock {\em Invariants of rational transformations and algebraic entropy},
  volume 219 of {\em AMS Contemporary Mathematics}, 1997.
\newblock pp 233--240.

\bibitem[MSS18]{Matsuzawa2018}
Yohsuke Matsuzawa, Kaoru Sano, and Takahiro Shibata.
\newblock Arithmetic degrees and dynamical degrees of endomorphisms on
  surfaces.
\newblock {\em Algebra Number Theory}, 12(7):1635--1657, 2018.

\bibitem[MW]{Matsuzawa}
Yohsuke Matsuzawa and Long Wang.
\newblock {A}rithmetic degrees and {Z}ariski dense orbits of cohomologically
  hyperbolic maps.
\newblock arXiv:2212.05804.

\bibitem[MZ22]{Meng2022}
Sheng Meng and De-Qi Zhang.
\newblock Kawaguchi-{S}ilverman conjecture for certain surjective
  endomorphisms.
\newblock {\em Doc. Math.}, 27:1605--1642, 2022.

\bibitem[MZ23]{Meng2023}
Sheng Meng and De-Qi Zhang.
\newblock {A}dvances in the equivariant minimal model program and their
  applications in complex and arithmetic dynamics.
\newblock arXiv:2311.16369, 2023.

\bibitem[Ngu06]{Nguyen2006}
Vi\^{e}t-Anh Nguy\^{e}n.
\newblock Algebraic degrees for iterates of meromorphic self-maps of {${\Bbb
  P}^k$}.
\newblock {\em Publ. Mat.}, 50(2):457--473, 2006.

\bibitem[Ogu09]{Oguiso2009}
Keiji Oguiso.
\newblock A remark on dynamical degrees of automorphisms of hyperk\"{a}hler
  manifolds.
\newblock {\em Manuscripta Math.}, 130(1):101--111, 2009.

\bibitem[Ogu10]{Oguiso2010}
Keiji Oguiso.
\newblock The third smallest {S}alem number in automorphisms of {$K3$}
  surfaces.
\newblock In {\em Algebraic geometry in {E}ast {A}sia---{S}eoul 2008},
  volume~60 of {\em Adv. Stud. Pure Math.}, pages 331--360. Math. Soc. Japan,
  Tokyo, 2010.

\bibitem[Ogu14]{Oguiso2014}
Keiji Oguiso.
\newblock Some aspects of explicit birational geometry inspired by complex
  dynamics.
\newblock In {\em Proceedings of the {I}nternational {C}ongress of
  {M}athematicians---{S}eoul 2014. {V}ol. {II}}, pages 695--721. Kyung Moon Sa,
  Seoul, 2014.

\bibitem[OT15]{Oguiso2015}
Keiji Oguiso and Tuyen~Trung Truong.
\newblock Explicit examples of rational and {C}alabi-{Y}au threefolds with
  primitive automorphisms of positive entropy.
\newblock {\em J. Math. Sci. Univ. Tokyo}, 22(1):361--385, 2015.

\bibitem[OY20]{Oguiso2020}
Keiji Oguiso and Xun Yu.
\newblock Minimum positive entropy of complex {E}nriques surface automorphisms.
\newblock {\em Duke Math. J.}, 169(18):3565--3606, 2020.

\bibitem[Pet15]{Petsche2015}
Clayton Petsche.
\newblock On the distribution of orbits in affine varieties.
\newblock {\em Ergodic Theory Dynam. Systems}, 35(7):2231--2241, 2015.

\bibitem[Res17]{Reschke2017}
Paul Reschke.
\newblock Salem numbers and automorphisms of abelian surfaces.
\newblock {\em Osaka J. Math.}, 54(1):1--15, 2017.

\bibitem[RG71]{Raynaud1971}
Michel Raynaud and Laurent Gruson.
\newblock Crit\`eres de platitude et de projectivit\'{e}. {T}echniques de
  ``platification'' d'un module.
\newblock {\em Invent. Math.}, 13:1--89, 1971.

\bibitem[RS97]{Russakovskii1997}
Alexander Russakovskii and Bernard Shiffman.
\newblock Value distribution for sequences of rational mappings and complex
  dynamics.
\newblock {\em Indiana Univ. Math. J.}, 46(3):897--932, 1997.

\bibitem[SB21]{Shepherd-Barron2021}
N.~I. Shepherd-Barron.
\newblock Some effectivity questions for plane {C}remona transformations in the
  context of symmetric key cryptography.
\newblock {\em Proc. Edinb. Math. Soc. (2)}, 64(1):1--28, 2021.

\bibitem[SC18]{Silverman2018}
Joseph~H. Silverman and Gregory~S. Call.
\newblock Degeneration of dynamical degrees in families of maps.
\newblock {\em Acta Arith.}, 184(2):101--116, 2018.

\bibitem[Sib99]{Sibony1999}
Nessim Sibony.
\newblock Dynamique des applications rationnelles de {$\bold P^k$}.
\newblock In {\em Dynamique et g\'{e}om\'{e}trie complexes ({L}yon, 1997)},
  volume~8 of {\em Panor. Synth\`eses}, pages ix--x, xi--xii, 97--185. Soc.
  Math. France, Paris, 1999.

\bibitem[Sil14]{Silverman2014}
Joseph~H. Silverman.
\newblock Dynamical degree, arithmetic entropy, and canonical heights for
  dominant rational self-maps of projective space.
\newblock {\em Ergodic Theory Dynam. Systems}, 34(2):647--678, 2014.

\bibitem[Son23]{Song2023}
Jiarui Song.
\newblock A high-codimensional arithmetic {S}iu's inequality and its
  application to higher arithmetic degrees.
\newblock arXiv:2306.11591, 2023.

\bibitem[SV22]{Shuddhodan2022}
K.~V. Shuddhodan and Yakov Varshavsky.
\newblock The {H}rushovski-{L}ang-{W}eil estimates.
\newblock {\em Algebr. Geom.}, 9(6):651--687, 2022.

\bibitem[TCD]{Tien-CuongDinha}
Tuyen Trung~Truong Tien-Cuong~Dinh, Viet-Anh~Nguyen.
\newblock Equidistribution for meromorphic maps with dominant topological
  degree.
\newblock arXiv:1303.5992.

\bibitem[Tru20]{Truong2020}
Tuyen~Trung Truong.
\newblock Relative dynamical degrees of correspondences over a field of
  arbitrary characteristic.
\newblock {\em J. Reine Angew. Math.}, 758:139--182, 2020.

\bibitem[Tru24]{Truong2016}
Tuyen~Trung Truong.
\newblock {R}elations between dynamical degrees, {W}eil's {R}iemann hypothesis
  and the standard conjectures.
\newblock {\em Commentarii Mathematici Helvetici}, 2024.
\newblock published on line.

\bibitem[Ueh16]{Uehara2016}
Takato Uehara.
\newblock Rational surface automorphisms with positive entropy.
\newblock {\em Ann. Inst. Fourier (Grenoble)}, 66(1):377--432, 2016.

\bibitem[Ure18]{Urech}
Christian Urech.
\newblock Remarks on the degree growth of birational transformations.
\newblock {\em Math. Res. Lett.}, 25(1):291--308, 2018.

\bibitem[Wan23]{Wang2022}
Long Wang.
\newblock {Periodic points and arithmetic degrees of certain rational
  self-maps}.
\newblock {\em Journal of the Mathematical Society of Japan}, pages 1 -- 26,
  2023.

\bibitem[Xie15]{Xie2015}
Junyi Xie.
\newblock Periodic points of birational transformations on projective surfaces.
\newblock {\em Duke Math. J.}, 164(5):903--932, 2015.

\bibitem[Xie23a]{Xie2023a}
Junyi Xie.
\newblock {A}round the dynamical mordell-lang conjecture.
\newblock arXiv:2307.05885, 2023.

\bibitem[Xie23b]{Xie2023}
Junyi Xie.
\newblock Remarks on algebraic dynamics in positive characteristic.
\newblock {\em J. Reine Angew. Math.}, 797:117--153, 2023.

\bibitem[XZ24]{Xu2024}
Disheng Xu and Jiesong Zhang.
\newblock {O}n holomorphic partially hyperbolic systems.
\newblock arXiv:2401.04310, 2024.

\bibitem[Yom87]{Yomdin1987}
Y.~Yomdin.
\newblock Volume growth and entropy.
\newblock {\em Israel J. Math.}, 57(3):285--300, 1987.

\end{thebibliography}
\end{document}